\documentclass[12pt]{amsart}

\usepackage{amssymb} 
\usepackage{amsthm}
\usepackage{amscd}
\usepackage{multicol}
\usepackage[dvips]{color}
\newtheorem{thm}{Theorem}[section]
\newtheorem{lem}[thm]{Lemma}
\newtheorem{lem-dfn}[thm]{Lemma-Definition}
\newtheorem{prop}[thm]{Proposition}
\newtheorem{cor}[thm]{Corollary}

\theoremstyle{definition}
\newtheorem{defn}[thm]{Definition}

\newtheorem{ex}[thm]{Example}

\newtheorem{quest}[thm]{Question}

\newtheorem*{acknowledgement}{Acknowledgement}
\newtheorem{not-prop}[thm]{Notation-Proposition}
\newtheorem{defn-prop}[thm]{Definition-Proposition}
\newtheorem{conv}[thm]{Convention}

\theoremstyle{remark}

\newtheorem{rem}[thm]{Remark}
\numberwithin{equation}{section}
\newcommand{\thmref}[1]{Theorem~\ref{#1}}
\newcommand{\lemref}[1]{Lemma~\ref{#1}}

\newcommand{\proref}[1]{Proposition~\ref{#1}}
\newcommand{\dpref}[1]{Definition-Proposition~\ref{#1}}
\newcommand{\remref}[1]{Remark~\ref{#1}}

\newcommand{\exref}[1]{Example~\ref{#1}}

\newcommand{\sref}[1]{Section~\ref{#1}}
\newcommand{\ssref}[1]{Subsection~\ref{#1}}

\DeclareMathOperator{\Spec}{Spec}
\DeclareMathOperator{\spec}{Spec}

\DeclareMathOperator{\supp}{Supp}
\DeclareMathOperator{\Ker}{Ker}
\DeclareMathOperator{\Image}{Im}
\DeclareMathOperator{\Img}{Im}
\DeclareMathOperator{\Coker}{Coker}
\DeclareMathOperator{\cd}{CD}

\DeclareMathOperator{\chara}{char}

\DeclareMathOperator{\pol}{Pol}
\DeclareMathOperator{\Fr}{F}
\DeclareMathOperator{\di}{div}
\DeclareMathOperator{\divv}{div}%KW0430

\DeclareMathOperator{\type}{type}

\newcommand{\m}{\mathfrak m}

\newcommand{\PP}{\mathbb P}

\newcommand{\bbA}{\mathbb A}
\newcommand{\Z}{\mathbb Z}
\newcommand{\Q}{\mathbb Q}

\newcommand{\cF}{\mathcal F}

\newcommand{\cO}{\mathcal O}

\newcommand{\cU}{\mathcal U}

\newcommand{\Ann}{\mathrm{Ann}}
\newcommand{\coh}{Z_{\mathrm{coh}}}

\renewcommand{\:}{\colon}

\newcommand{\ol}[1]{\overline {#1}}

\newcommand{\fl}[1]{\left\lfloor #1 \right\rfloor}
\newcommand{\ce}[1]{\left\lceil #1 \right\rceil}
\newcommand{\gi}[1]{\langle #1 \rangle}

\newcommand{\defset}[2]{{\left\{#1\,\left| \,#2 \right. \right\}}}

\newcommand{\al}{{\alpha}}

\begin{document}
\title{A Geometric description of almost Gorensteinness for two-dimensional normal singularities}

%%%%%%%%%%%%%%%%%%%%%%%%%%%%%%%%%%%%%%%%%%%%%%%{\color{red} %%%%%%%%%%%%%%%% 
%%   Information for first author
%%%%%%% %%%%%%% %%%%%%% %%%%%%% %%%%%%% %%%%%%% %%%%%%% 
\author{Tomohiro Okuma}
\address[Tomohiro Okuma]{Department of Mathematical Sciences, 
Yamagata University,  Yamagata, 990-8560, Japan.}
\email{okuma@sci.kj.yamagata-u.ac.jp}
%%%%%%%%%%%%%%%%%%%%%%%%%%%%%%%%%%%%%%%%%%%%%%%%%%%%%%%%%%%%%%% 
%%    Information for second author
%%
\author{Kei-ichi Watanabe}
\address[Kei-ichi Watanabe]{Department of Mathematics, College of Humanities and Sciences, 
Nihon University, Setagaya-ku, Tokyo, 156-8550, Japan and 
Organization for the Strategic Coordination of Research and Intellectual Properties, Meiji University
}
\email{watnbkei@gmail.com}
%%%%%%%%%%%%%%%%%%%%%%%%%%%%%%%%%%%%%%%%%%%%%%%%%%%%%%%%%%%%%%% 
%%    Information for third author
%%%%%%% %%%%%%% %%%%%%% %%%%%%% %%%%%%% %%%%%%% %%%%%%% 
\author{Ken-ichi Yoshida}
\address[Ken-ichi Yoshida]{Department of Mathematics, 
College of Humanities and Sciences, 
Nihon University, Setagaya-ku, Tokyo, 156-8550, Japan}
\email{yoshida.kennichi@nihon-u.ac.jp}
%%%%%%%%%%%%%%%%%%%%%%%%%%%%%%%%%%%%%%%%%%%%%%%%%%%%%%%%%%%%%%% 
\thanks{TO was partially supported by JSPS Grant-in-Aid 
for Scientific Research (C) Grant Number 21K03215.
KW  was partially supported by JSPS Grant-in-Aid 
for Scientific Research (C) Grant Number 23K03040.
KY was partially supported by JSPS Grant-in-Aid 
for Scientific Research (C) Grant Number 19K03430.}
%%%%%%%%%%%%%%%%%%%%%%%%%%%%%%%%%%%%%%%%%%%%%%%%%%%%%%%%%%%%%%% 
\keywords{Almost Gorenstein local ring, elliptic singularity, rational singularity, $p_g$-ideal, two-dimensional normal domain}
\subjclass[2020]{Primary: 13H10; Secondary: 13G05, 14B05, 14J17}
%\subjclass[2020]{13G05, 14J17,13H10, 14J27}
%%%%%%%%%%%%%%%%%%%%%%%%%%%%%%%%%%%%%

%%% Abstract
\begin{abstract} 
Let $A$ be an excellent two-dimensional normal local ring containing an algebraically closed field. Then $A$ is 
called an elliptic singularity %}
if $p_f(A)=1$, where $p_f$ denotes the fundamental genus. 
On the other hand, the concept of almost Gorenstein rings was introduced by 
Barucci and Fr\"oberg %\cite{BF} 
for one-dimensional local rings and %KW0620 it 
generalized by Goto, Takahashi and Taniguchi %\cite{GTT} 
to higher dimension.\par
In this paper, we describe almost Gorenstein rings in geometric language using 
resolution of singularities and give criterions to be almost Gorenstein. 
In particular, we show that  elliptic singularities are almost Gorenstein. 
Also, for every integer $g\ge 2$, we provide examples of singularities that is 
almost Gorenstein (resp.  not almost Gorenstein) with $p_f(A)=g$.%}
\par
We also provide several examples of determinantal singularities associated with $2\times 3$ matrices, which include both almost Gorenstein singularities and non-almost Gorenstein singularities.%}%TO0628 
%We also provide several determinantal singularities associated with $2\times 3$ matrices{\color{blue}, which are almost Gorenstein.}%KW0620 
\end{abstract}

\maketitle

\section{Introduction}

The notion of almost Gorenstein local rings was introduced by Barucci and Fr\"oberg \cite{BF} for one-dimensional local rings
and %KW0620 it 
 has been generalized by Goto, Takahashi, and Taniguchi \cite{GTT} for Cohen-Macaulay local rings having canonical modules of any dimension. 
Namely, when $(A,\m)$ is a Cohen-Macaulay 
%{\color{red}
 local %}
 ring of dimension $d$, 
 we call $A$  {\bf almost Gorenstein} if there exists an  element 
 $\omega \in K_A$ such that the $A$-module $U = K_A/ \omega A$ satisfies the equality
\[
\mu(U) = e_0(U),
\]   
%}
where  $\mu(U) = \dim_{A/\m}( U /\m U)$ is the number of minimal generators of $U$ and 
$e_0(U)$ denotes the multiplicity of $U$. An $A$-module which satisfies this property is 
called an % {\bf Ulrich $A$-module}.
 {\em Ulrich $A$-module}.
In other words, a Cohen-Macaulay $A$-module $U$ of dimension $d-1$
 is an Ulrich $A$-module  if
for some parameter system $(x_1, x_2, \ldots, x_{d-1})$ of $U$, we have the equality 
$(x_1, x_2, \ldots, x_{d-1})U = \m U$. 

The aim of this paper is to determine 
whether a normal two-dimensional ring $A$ is  
almost Gorenstein using resolutions of singularities of $A$.

Let $(A, \m)$ be an excellent two-dimensional normal local ring containing an algebraically closed field.
In \cite[\S 11]{GTT}, it is proved that $A$ is almost Gorenstein if $\m$ is a $p_g$-ideal (cf. \cite{PWY}); this fact implies that a rational singularity is almost Gorenstein. 

In this article, we prove that $A$ is almost Gorenstein if it is an elliptic singularity (\thmref{t:elliptic}),
by using cohomology of sheaves on a resolution of singularities.
Let $X \to \spec (A)$ be a resolution of singularities. 
We define that $A$ is an {\em elliptic singularity} if $p_f(A)=1$, where 
$p_f(A)$ denotes the arithmetic genus of the Artin's fundamental cycle $Z_f$ on  $X$.
The concept of elliptic singularities was introduced by Wagreich \cite{wag.ell} and have been studied  by 
 many authors: Laufer \cite{la.me}, Yau \cite{yau.normal}, \cite{yau.max}, Tomari  \cite{tomari.ell}, N{\'e}methi \cite{nem.ellip},  Okuma \cite{o.numGell}, Nagy--N{\'e}methi \cite{nn.AbelIII}.
Our argument is based on the analysis of cohomology of coherent sheaves on $X$ related with canonical divisor $K_X$ and the 
 cycle on $X$ which represents the maximal ideal $\m$. 
In the proof of the main result, a property of the cohomological cycle for elliptic singularities (Konno \cite{Konno-coh}) plays an important role.
Our method is also useful to prove %KW0620  for the proof of
that $A$ is almost Gorenstein if $\m$ is a $p_g$-ideal 
(\thmref{t:pg}). 
 In particular, this gives a new proof to show that a rational singularity is almost
 Gorenstein.

This paper is organized as follows.

In \sref{s:pre}, we recall the terminology and basics of cycles on a resolution of singularities,  and then introduce a commutative diagram that connects the canonical sheaf on the resolution space with an $A$-module $U$ which determines 
 almost Gorenstein property of $A$.

Let $K_X$ be the canonical divisor of $X$. Recall that $H^0(\cO_X(K_X))\subset K_A$ and 
$\ell_A (K_A/ H^0(\cO_X(K_X)) = p_g(A)$. Namely, $H^0(\cO_X(K_X)) = K_A$ holds if 
and only if  $A$ is a rational singularity. 
In general, we can find a resolution $X$ and a cycle $C_X$ on $X$ so that  
$H^0(\cO( K_X+ C_X)) = K_A$ and $\cO( K_X+ C_X)$ is generated by global sections. 
In the following, we %always 
assume that $K_X+ C_X$ satisfies these 
conditions.

%\bigskip

We introduce the most important sheaf $\cU$ on $X$ defined by the exact sequence 
\[0 \to  \cO_X \overset{\times \omega}\longrightarrow \cO_X(K_X+C_X) \to  \cU \to 0, \]
where  $\omega \in K_A = H^0(X, \cO_X(K_X+C_X))$ 
is a general element of $K_A$. 
We discuss fundamental properties of sheaves related to this exact sequence. 
We have that $ U \subset H^0(\cU)$ and $\cU$ is always
generated by global sections. 

In \sref{s:Ugg}, we prove that  if  $H^0(\cU) = U$ holds, 
then  $A$ is an almost Gorenstein ring.
Moreover, %Incidentally, 
if $A$ is a rational singularity or
an elliptic singularity, then we can verify that $H^0(\cU) = U$ and 
can conclude that 
{\bf rational singularities and elliptic singularities are almost Gorenstein.}
\par 

In \sref{s:mult}, we provide a characterization of a singularity being almost Gorenstein
 in terms of $\cU$. 
 Among several examples,  we can prove that a singularity is almost Gorenstein
  if its maximal ideal is a $p_g$-ideal.
    
In \sref{s:Ell}, we give some examples of elliptic singularities  and 
show ring-theoretic properties of those singularities.

In \sref{s:example}, we consider cone singularities $R(C,D)$ associated with a nonsingular curve $C$ of genus $g \ge 2$ and a divisor $D$ on $C$ with $\deg D>0$. 
Note that  $p_f(A) = g$ for these singularities.
 We can provide many examples of both almost Gorenstein rings and not almost 
Gorenstein rings of this class.
In particular, in the case $D=K_C+p$,  where $p$ is a point of $C$, we prove that $R(C,D)$ is almost Gorenstein if and only if $C$ is hyperelliptic.
Also,  we show that $R(C,D)$ is not almost Gorenstein if $\deg D> 2 \deg K_C$.
These examples also show that
the almost Gorenstein property cannot be determined only by the pair $(g, \deg D)$.

In \sref{s:det}, we provide determinantal two-dimensional normal singularities associated with $2\times 3$ matrices of the following types: 
(1) $A$ is non-rational and $\m$ is $p_g$-ideal, 
(2) %$A$ is elliptic and non-Gorenstein, (3) 
$A$ is non-elliptic 
almost Gorenstein and $\m$ is not $p_g$-ideal, and (3) $A$ is not almost Gorenstein.  

\begin{acknowledgement}
The authors are grateful to Professor Masataka Tomari for letting us know his result on elliptic singularities.
\end{acknowledgement}

\section{Preliminaries}\label{s:pre}

Throughout this paper, let $(A,\m)$ be an excellent two-dimensional normal local ring containing an algebraically closed field $k \cong A/\m$,  and let 
$\pi\:X \to \spec (A)$ a resolution of singularities with exceptional 
divisor $E:=\pi^{-1}(\m)$. 
Let $E=\bigcup_{i=1}^n E_i$ be the decomposition into the irreducible components, and let $K_X$ denote the canonical divisor on $X$.

\subsection{Almost Gorenstein local rings}\label{ss:AGL}
\phantom{AAAAAAA}\par 

For a finitely generated 
$A$-module $M$, we denote by $e_0(M)$ the multiplicity of $M$ with respect to $\m$ and by $\mu(M)$ the number of elements in a minimal system of generators for $M$; 
namely, if $\dim M =s$, 
\[
e_0(M) = \lim_{n\to \infty} \dfrac{\ell_A(M/\m^n M) s!}{n^s}, \quad 
\mu(M)=\ell_A(M/\m M).
\]
We say that $M$ is an {\em Ulrich module} if $M$ is a Cohen-Macaulay module with $e_0(M)=\mu(M)$.
 If $(x_1, \ldots, x_s)$ is a minimal reduction of the maximal ideal of 
$A/\Ann_A(M)$, then $M$ is an Ulrich $A$-module if and only if $(x_1, \ldots, x_s)M= \m M$ 
holds. 

Let $K_A$ denote the canonical module of $A$ and let $\type(A)=\mu(K_A)$.

\begin{defn}
[cf. {\cite[Definition 3.3]{GTT}}]
$A$ is said to be {\em almost Gorenstein}  if there is an exact sequence of $A$-modules
\[
0 \to A \to K_A \to U \to 0
\]
such that 
$U$ is an Ulrich $A$-module. 
 Note that the exactness implies that $U=0$ or $U$ is a Cohen-Macaulay $A$-module of $\dim U =1$ (cf. \cite[Lemma 3.1]{GTT}).
\par 
A Gorenstein singularity is almost Gorenstein by definition.
\end{defn}

Any homomorphism $A\to K_A$ is given by an element $\omega \in K_A$. If we take a general element so that $\omega\not \in \m K_A$, then  $\mu(K_A/\omega A) = \mu (K_A)-1=\type(A)-1$.

\subsection{Cycles and genera}\label{ss: cycle&g}
\phantom{AAAAAAA}\par 

A divisor $D$ on $X$  with $\supp(D)\subset E$ is called a {\em cycle}.
We may consider a cycle $D>0$ as a scheme with structure sheaf $\cO_D := \cO_X/\cO_X(-D)$.
We write $\chi(C)=\chi(\cO_C):=h^0(\cO_C)-h^1(\cO_C)$, where $h^i(\cF)$ denotes $\ell_A(H^i(\cF))$.
A divisor $L$ on $X$ is said to be {\em nef} (resp. {\em anti-nef}) if $LE_i \ge 0$ (resp. $LE_i\le 0$) for all $E_i \subset E$.
It is known that anti-nef cycles are effective
%{\color{red}
and the support of a non-zero anti-nef cycle is $E$.
%}
 There exists the minimum $Z_f$ among the anti-nef cycles $D>0$; 
the cycle $Z_f$ is called the {\em fundamental cycle} (\cite{artin.rat}).
The {\em fundamental genus} $p_f(A)$ of $A$ is defined by $p_f(A)=1-\chi(Z_f)$, and the {\em geometric genus} $p_g(A)$ of $A$ is defined by $p_g(A)=h^1(\cO_X)$; these are independent of the choice of the resolution.
We have $p_f(A) \le p_g(A)$.
While the arithmetic genus is determined by the resolution graph, the geometric genus is not. %}
The local ring $A$ is said to be {\em rational} if $p_g(A)=0$.
It is known that $A$ is a rational singularity if and only if $p_f(A)=0$ (see \cite{artin.rat}).
If $I \subset A$ is an integrally closed $\m$-primary ideal, then 
there exists a resolution $X \to \Spec A$ and a cycle $Z>0$ on $X$ 
such that $I \cO_X = \cO_X(-Z)$ and $I=H^0(X,\cO_X(-Z))$. In this case, we say that $I$ is {\em represented} by $Z$ on $X$.  

\subsection{Fundamental exact sequences}
\label{ss:exactseq}
\phantom{AAAAAAA}\par 

Let $L$ be a divisor on $X$.
We say that $L$  or an invertible sheaf $\cO_X(L)$ has {\em no fixed component} if the restriction $H^0(\cO_X(L))\to H^0(\cO_{E_i}(L))$ is non-trivial for every $E_i\subset E$.
If $L$ has no fixed component, then the base locus of $H^0(\cO_X(L))$ does not contain any $E_i$ and a general element of $H^0(\cO_X(L))$ generates $\cO_X(L)$ at general points of $E$.

\bigskip

%{\color{red}
{\bf In this subsection, we assume the following.}%KW0813

\begin{enumerate}
\item Assume that  $L$ has no fixed component, and let $\omega\in H^0(\cO_X(L))$ be a general element and $\cU=\cO_X(L)/\omega\cO_X$.
\item Let $D>0$ be any cycle on $X$.
\item We write $\cU(-D)=\cU\otimes_{\cO_X}\cO_X(-D)$ and $\cU_D=\cU\otimes_{\cO_X}\cO_D$.
\end{enumerate}
We call a subscheme of $X$ of pure dimension one a {\bf curve}.
%}

\begin{lem}\label{l:inj}
We have the following.
\begin{enumerate}
\item The support of $\cU$ is an affine curve 
and  $\supp(\cU_D)$ is a finite set.
\item The morphism $\cO_D \to \cO_D(L)$ induced by $\cO_X \xrightarrow{\times \omega} \cO_X(L)$ is injective.
\item For any divisor $L'$ on $X$, 
we have $h^1(\cO_X(L')) \ge h^1(\cO_X(L+L'))$.
\end{enumerate}
\end{lem}
%}
\begin{proof}
Since $L$ has no fixed component,
$\omega$ generates $\cO_X(L)$ at general points  of %each $E_i$
$E$,
we see that $\supp(\cU)$ does not contain any $E_i$ and that 
$\Ker(\cO_D \to \cO_D(L))$ should be a torsion submodule of $\cO_D$; however, $\cO_D$ is torsion free.
%{\color{red} 
Clearly,  $\supp(\cU_D) = \supp(\cU)\cap \supp(D)$.
%}
Thus we have (1) and (2).

By tensoring the exact sequence
\[
0 \to \cO_X \xrightarrow{\times \omega} \cO_X(L) \to \cU \to 0
\]
with $\cO_X(L')$, since 
 $\cU$ is a coherent sheaf on an affine scheme, 
we have a surjection $H^1(\cO_X(L')) \to H^1(\cO_X(L+L'))$ (cf. \cite[\S 2.6]{OWYgood}).
\end{proof}

We denote by $\cd(L,\omega,D)$ the commutative diagram shown in Table \ref{tab:CD}:

\begin{table}[htb]
\[
\begin{CD}
@. 0 @. 0 @. 0 @. \\
@. @VVV @VVV @VVV \\
0 @>>> \cO_X(-D) @>{\times \omega}>> \cO_X(L-D) @ >>> \cU (-D) @>>> 0 \\
@. @VVV @VVV @VVV \\
0 @>>> \cO_X @>{\times \omega}>> \cO_X(L) @ >>> \cU @>>> 0  \\
@. @VVV @VVV @VVV \\
0 @>>> \cO_D @>{\times \omega}>>\cO_D(L) @ >>> \cU_D @>>> 0 \\
@. @VVV @VVV @VVV \\
@. 0 @. 0 @. 0 @. \\
\end{CD}
\]

\smallskip

\caption{\label{tab:CD} $\cd(L,\omega,D)$}
\end{table}

\begin{lem}\label{l:cdU}
We have the following.
\begin{enumerate}
\item All vertical and horizontal sequences in $\cd(L,\omega,D)$ are exact.
\item We have $H^1(\cU (-D))=H^1(\cU)=H^1(\cU_D)=0$ and 
$\ell_A(H^0(\cU)/H^0(\cU(-D)))=h^0(\cU_D)=LD$.
\end{enumerate}
\end{lem}
\begin{proof}
(1) 
The morphism $\cO_D \to \cO_D(L)$ is injective by \lemref{l:inj}.
This implies that $\cU(-D)\to \cU$ is also injective.
We can easily check the other exactness.

(2) Since the support of $\cU$ is an affine curve by \lemref{l:inj}, these cohomology groups vanish. 
From the third horizontal exact sequence in  $\cd(L,\omega,D)$, we have 
\[
h^0(\cU_D)=\chi(\cU_D)=\chi(\cO_D(L)) - \chi(\cO_D)
=LD.
\qedhere
\]
\end{proof}

The arguments in \cite[Proposition 2.6]{OWYgood} yields the following 
%{\color{red}
lemma that plays an important role in \sref{s:Ugg} and \sref{s:mult}. 
%}

\begin{lem}
\label{l:2ndfund}
Let $L_1$ and $L_2$ be divisors on $X$ such that $\cO_X(L_1)$ and $\cO_X(L_2)$ are generated.
%{\color{red}
 Take general elements $f_1 \in H^0(\cO_X(L_1))$ and 
$f_2\in H^0(\cO_X(L_2))$.
Then we have the following exact sequence$:$
\begin{equation} \label{eq: 2ndfund}
0 \to \cO_X \xrightarrow {(f_1,f_2)}
 \cO_X(L_1) \oplus
\cO_X(L_2)
\xrightarrow{\genfrac{(}{)}{0pt}{}{-f_2}{f_1}}
\cO_X(L_1+L_2) \to 0.
\end{equation}
%}
\begin{enumerate} 
\item If $h^1(\cO_X(L_1)) = p_g(A)$ or $h^1(\cO_X(L_2)) = p_g(A)$ holds, then the 
induced mapping 
$ H^0(\cO_X(L_1))  \oplus H^0( \cO_X(L_2)) 
\xrightarrow{\genfrac{(}{)}{0pt}{}{-f_2}{f_1}}
H^0( \cO_X(L_1+L_2))$ is surjective and hence we have 
$ H^0( \cO_X(L_1+L_2)) = f_1 H^0( \cO_X(L_2)) + f_2 H^0(\cO_X(L_1)) $.
\item If $h^1(\cO_X(L_1)) = p_g(A)$ holds, then we have $h^1(\cO_X(L_1+L_2)) = h^1(\cO_X(L_2))$. 
\end{enumerate}
\end{lem}
\begin{proof}
%{\color{red}
We only explain (1) and (2).
Assume $h^1(\cO_X(L_1)) = p_g(A)$. Then since 
$h^1(\cO_X(L_1+L_2))\le h^1(\cO_X(L_2))$
 by \lemref{l:inj} (3), 
we must have 
\begin{equation}\label{eq:h^1}
h^1(\cO_X) - (h^1(\cO_X(L_1))+ h^1(\cO_X(L_2))) + h^1(\cO_X(L_1+L_2)) = 0
\end{equation}
 and hence 
the map 
\[H^0(\cO_X(L_1))  \oplus H^0( \cO_X(L_2)) 
\xrightarrow{\genfrac{(}{)}{0pt}{}{-f_2}{f_1}}
H^0( \cO_X(L_1+L_2))\] is surjective. 
The claim (2) follows from the equality \eqref{eq:h^1}.
%}
\end{proof}

\subsection{The cohomological cycle and the canonical module}\label{ss:coh}
\phantom{AAAAAAA}\par 

There exists the minimum $\coh\ge 0$ of the cycles $C\ge 0$ on $X$ satisfying $h^1(\cO_C)=p_g(A)$; where $h^1(\cO_C)=0$ for $C=0$ (cf. \cite[4.8]{chap}).
We call $\coh$ the {\em cohomological cycle} on $X$.
Note that  $A$ is Gorenstein if and only if $K_X$ is linearly equivalent to a cycle because $X\setminus E = \spec(A)\setminus\{\m\}$, and that 
if $A$ is Gorenstein and $\pi$ is the minimal resolution, then $K_X\sim - \coh$ (\cite[4.20]{chap}), where $Z\sim Z'$ denotes that the divisors $Z$ and $Z'$ are 
linearly equivalent, 
%{\color{red}
or equivalently, %}
$\cO_X(Z) \cong \cO_X(Z')$.

\begin{lem}\label{l:coh}
 Assume that  $L$ has no fixed component and 
 that a cycle $D$ satisfies $D\ge \coh$. Then $H^1(\cO_X(L)) \cong H^1(\cO_{D}(L))$.
\end{lem}
\begin{proof}
Since we have the surjections $\cO_X(L) \to \cO_D(L) \to \cO_{\coh}(L)$, we may assume that $D=\coh$.
By \lemref{l:cdU}, all vertical and horizontal sequences in the following commutative diagram are exact:
\[
\begin{CD}
 H^0(\cU(-D)) @>>> H^1(\cO_X(-D))  @>>>  H^1(\cO_X(L-D)) @>>>  0 \\
@VVV @VV{\alpha}V @VVV \\
  H^0(\cU) @>>> H^1(\cO_X) @>>> H^1(\cO_X(L)) @>>> 0 \\
@VVV @VV{\beta}V @VV{\gamma}V \\
 H^0(\cU_D) @>>> H^1(\cO_{D}) @>>> H^1(\cO_{D}(L))  @>>> 0 \\
@VVV @VVV @VVV \\
 0 @. 0 @. 0 @. \\
\end{CD}
\]
Since $\beta$ is an isomorphism by the assumption $D=\coh$,  $\alpha $ is trivial, and thus $\gamma$  is also an isomorphism.
\end{proof}

Recall that $K_A$ denotes the canonical module of $A$.
Since $K_A$ is reflexive and $X\setminus E \cong \spec(A) \setminus \{\m\}$, we have $K_A=H^0(X\setminus E, \cO_X(K_X))$. 
%Thus we have the exact sequence
%\[
%0 \to H^0(\cO_X(K_X)) \to K_A \to H^1_E(\cO_X(K_X)) \to H^1(\cO_X(K_X)).
%\]
%By the duality and the Grauert-Riemenschneider vanishing theorem 
% (cf. \cite{gi}), 
%we have the following.
\begin{prop}
[cf. {\cite[Theorem 3.4 and 3.5]{la.rat}, \cite[4.4]{karras} }]
\label{p:KArat}
We have 
\[
p_g(A)=\ell_A(K_A/H^0(\cO_X(K_X)).
\]
In particular, $K_A=H^0(\cO_X(K_X))$ if $A$ is a rational singularity.
\end{prop}

\begin{defn}
Let $C_X$ denote the minimum of the cycles $C$ such that $K_A=H^0(X,\cO_X(K_X+C))$.
\end{defn}

\begin{rem}\label{r:CX}
We have the following.
\begin{enumerate}
\item If $C_X=C_1-C_2$, where $C_1$ and $C_2$ are effective cycles without common components, then $C_1=\coh$ (cf. \cite[4.3]{ORWY}).
\item $\cO_X(K_X+C_X)$ has no fixed component.
In fact, if $H^0(\cO_X(K_X+C_X))\to H^0(\cO_{E_i}(K_X+C_X))$ is trivial for some $E_i$, then $H^0(\cO_X(K_X+C_X-E_i))=H^0(\cO_X(K_X+C_X))$; it contradicts the minimality of $C_X$. 
In particular, $K_X+C_X$ is nef.
\end{enumerate}
\end{rem}

There exists a $\Q$-divisor $Z_{K_X}$ supported in $E$ such that $K_X+Z_{Z_X}\equiv 0$, i.e., $(K_X+Z_{K_X})E_i= 0$ for every $E_i\subset E$. 
% Since $K_A = H^0(X\setminus E, \cO_X(K_X))$, 
%{\color{red} 
 As mentioned above,
%}
$A$ is Gorenstein if and only if $K_X$ is linearly equivalent to a cycle.
In fact, $A$ is Gorenstein if and only if $K_X\sim -Z_{K_X}$.  
The following proposition shows that $C_X=Z_{K_X}$ when $A$ is Gorenstein. %}

\begin{prop}\label{p:GCX}
 $A$ is Gorenstein if and only if $\cO_X(K_X+C_X)\cong \cO_X$.
\end{prop}
\begin{proof}
Assume that $A$ is Gorenstein. 
Since $K_X$ is linearly equivalent to a cycle, it follows from \remref{r:CX} (2) that there exists an anti-nef cycle $Z\ge 0$ such that $\cO_X(K_X+C_X)\cong \cO_X(-Z)$.  If $Z\ne 0$, then 
%{\color{red}
$Z\ge Z_f \ge E$
%}
 by the definition of $Z_f$ and $H^0(\cO_X(-Z))\subset \m$; it contradicts that $K_A\cong A$. %}
If $\cO_X(K_X+C_X)\cong \cO_X$, then $K_X$ is linearly equivalent to the cycle $-C_X$
 and hence $A$ is Gorenstein. 
\end{proof}

\begin{prop}\label{p:KAfree}
There exists a resolution $\pi\: X\to \spec(A)$ such that $\cO_X(K_X+C_X)$ is generated.
%{\color{red}
If $\phi\: \ol X \to X$ is the blowing up at a point with exceptional set $F$ and 
$\cO_X(K_X+C_X)$ is generated, then $\cO_{\ol X}(K_{\ol X}+C_{\ol X})$ is also generated and $C_{\ol X}=\phi^*C_X-F$.
%}
\end{prop}
\begin{proof}
Assume that $\cO_X(K_X+C_X)$ has base points.
Let 
$J \subset \cO_X$ be the ideal sheaf %}
of the base points; that is, $J$ satisfies $K_A \cO_X = J \cO_X(K_X+C_X)$. 
By \remref{r:CX} (2), $\supp(\cO_X/J)$ is a finite set.
Let $X' \to X$ be a log-resolution of the ideal $J$.
Then $\cO_{X'}(K_{X'}+C_{X'})$ is generated since 
$K_A\cO_{X'}$ is invertible and $H^0(\cO_{X'}(K_{X'}+C_{X'})) = K_A$.
%{\color{red}
The second claim follows from that $\phi^*\cO_X(K_X+C_X)=\cO_{\ol X}(K_{\ol X}-F+\phi^*C_X)$ is generated (cf. \remref{r:CX}).
%}
\end{proof}

\section{The case when $U = H^0(\cU)$; Cases when $A$ is rational or elliptic.}
\label{s:Ugg}

In this section we prove that $A$ is almost Gorenstein if 
%{\color{red}
we have $U = H^0(X, \cU)$ in the Table \ref{tab:KD} below, 
%}
 or, equivalently, $H^0(\al_0)$ is surjective. 
We can also show that several important classes of singularities including  rational and
 elliptic singularities satisfy this condition.  

\bigskip 

As we have shown in Proposition \ref{p:KAfree}, we can take a resolution of $\Spec(A)$ 
so that $\cO_X( K_X + C_X)$ is generated by $H^0(\cO_X( K_X + C_X)) = K_A$
and $\m\cO_X$ is invertible.
 In this section, 
we always assume that our resolution $X$ satisfies these conditions and $\m$ is represented by a cycle $Z$ on $X$; 
%}
that is $H^0(\cO_X( -Z)) = \m$ and $\m \cO_X= \cO_X(-Z)$.

Putting $L = K_X + C_X$ and $D = Z $ in Table \ref{tab:CD}, 
we get the %following 
%{\color{red}
diagram in Table \ref{tab:KD}. 
%}
Note that $L = K_X + C_X$ and $D = Z $ satisfy the conditions of \sref{ss:exactseq}.

%We can give an explicit condition about the choice of the element $\omega \in K_A$. 
%{\color{red}
Let us recall the key points of the choice of the element $\omega \in K_A$ (see \ssref{ss:exactseq}).

\begin{rem}\label{w:choice} We chose a \lq\lq general" element $\omega \in K_A$. But we notice that 
our argument depends only on the fact that the support of $\cU$ is affine. This is equivalent 
to say that $\omega$ generates $\cO_X(K_X+C_X)$ at the generic point of every irreducible exceptional 
curve $E_i$ on $X$.   
\end{rem}
%}

\begin{table}[htb]
\[
\begin{CD}
@. 0 @. 0 @. 0 @. \\
@. @VVV @VVV @VVV \\
0 @>>> \cO_X(-nZ) @>{\times \omega}>> \cO_X(K_X+C_X-nZ) @ >{\alpha_n}>> \cU (-nZ) @>>> 0 \\
@. @VVV @VVV @VVV \\
0 @>>> \cO_X @>{\times \omega}>> \cO_X(K_X+C_X) @ >{\alpha_0}>> \cU @>>> 0  \\
@. @VVV @VVV @VVV \\
0 @>>> \cO_{nZ} @>{\times \omega}>>\cO_{nZ}(K_X+C_X) @ >>> \cU_{nZ} @>>> 0 \\
@. @VVV @VVV @VVV \\
@. 0 @. 0 @. 0 @. \\
\end{CD}
\]
\smallskip

\caption{\label{tab:KD} $\cd(K_X+C_X,\omega,nZ)$}
\end{table} 

\begin{thm}\label{th;Ugg} If $U = H^0(X, \cU)$, or, equivalently, if $H^0(\alpha_0)$ is 
surjective, then $A$ is almost Gorenstein. 
\end{thm}

Before proving \thmref{th;Ugg}, we show a lemma which follows directly from Lemma \ref{l:2ndfund}; note that $H^0(\cO_X(-nZ))=\ol{\m^n}$, where $\ol{\m^n}$ denotes the integral closure of $\m^n$.

\begin{lem}\label{K-nZ}   
If  $h^1(\cO_X(K_X+C_X))= p_g(A)$ or $h^1( \cO_X(-Z)) = p_g(A)$ holds, then for 
every integer $n\ge 1$ we have $H^0(K_X+C_X - nZ) = \ol{\m^n} K_A$ and $\Image(\al_n) = 
\overline{\m^n}U$.  
\end{lem}

\begin{proof}[Proof of Theorem \ref{th;Ugg}]
In the following, we assume that $U = H^0(\cU)$, or, equivalently, 
 $H^0(\alpha_0) : H^0( \cO_X(K_X+C_X)) = K_A \to H^0(\cU)$
 is surjective. 
\par
 Note that, since $H^1(\cU)=0$, this condition 
 is   also equivalent to saying that 
  $H^1(\cO_X) \to H^1(\cO_X(K_X+C_X))$ is isomorphism. 
%{\color{red}
Hence we have $h^1(\cO_X(K_X+C_X))= p_g(A)$. 
%}
Then by Lemma \ref{l:2ndfund} we have $h^1(\cO_X(K_X+C_X-nZ))= h^1(\cO_X(-nZ))$
and since $H^1(\cU(-nZ))=0$, 
 $H^0(\alpha_n) : H^0(\cO_X(K_X+C_X-nZ)) \to H^0(\cU (-nZ))$ 
is surjective. 
%{\color{red}
Thus we get that the rows and columns of the following Table 3 are exact. 
%}
\begin{table}[htb]
\[
\begin{CD}
@. 0 @. 0 @. 0 @. \\
@. @VVV @VVV @VVV \\
0 @>>> H^0(\cO_X(-nZ)) @>{\times \omega}>> 
H^0(\cO_X(K_X+C_X-nZ))  @ >{a_n}>> H^0(\cU (-nZ))   @>>> 0 \\
@. @VVV @VVV @VVV \\
0 @>>> H^0(\cO_X)=A @>{\times \omega}>> H^0(\cO_X(K_X+C_X))=K_A @ >{a_0}>> H^0(\cU)=U @>>> 0  \\
@. @VVV @VVV @VVV \\
0 @>>> H^0(\cO_{nZ}) @>{\times \omega}>>H^0(\cO_{nZ}(K_X+C_X)) @ >>> H^0(\cU_{nZ}) @>>> 0 \\
@. @VVV @VVV @VVV \\
@. 0 @. 0 @. 0 @. \\
\end{CD}
\]
\smallskip

\caption{\label{tab:KZ} 
%{\color{red}
A diagram from %}
$\cd(K_X+C_X,\omega,nZ)$}
\end{table}
%{\color{red}
Also, by Lemma \ref{K-nZ} and since $U = K_A/ \omega A$, we have
 $H^0(\cU(-nZ)) = \ol{\m^n} U$ (cf. the diagram \eqref{eq:mU} below). 
Therefore, $U/ \overline{\m^n} U \cong H^0(\cU_{nZ})$. 
Hence we have 
\[
\ell_A ( U/ \overline{\m^n} U) = \ell_A (H^0(\cU_{nZ})) = n (K_X + C_X)\cdot Z
\]
by \lemref{l:cdU}.
Putting $n=1$, we have $\mu(U) = \ell_A( H^0(\cU_{Z}))$.
We also have 
\[
e_0(U) = \lim_{n\to \infty} \frac{\ell_A ( U/ \overline{\m^n} U)}{n} =(K_X + C_X)\cdot Z = \mu(U).
\qedhere
\] 
%}
%Thus we have shown that $A$ is almost Gorenstein. 
%}
%Hence $\ell_A( H^0(\cU_{Z})) =  \chi(\cO_Z(K_X+C_X)) - \chi(\cO_Z) = (K_X + C_X)\cdot Z$. 
%
%By the same reason,  since $H^0(\cU( -nZ)) = \overline{\m^n} U$,  we have 
%\[
%\ell_A ( U/ \overline{\m^n} U) = \ell_A (H^0_{nZ}(\cO_{nZ}(K_X+C_X)) = n (K_X + C_X)\cdot Z.
%\]
\end{proof} 

If $A$ is a rational singularity, then since 
%{\color{red} %$C_X=0$  and 
$H^1(\cO_X)=0$, 
%}
we have $U=H^0(\cU)$
and we have a proof of the following fact by geometric description.  

\begin{cor} 
If $A$ is a rational singularity, then $A$ is almost Gorenstein.
\end{cor} 

\subsection{Elliptic Singularities}\label{ss:ell}
\phantom{AAAAAAA}\par 

The local ring $A$ is said to be {\em elliptic} if  $p_f(A)=1$.
Note that $A$ is an elliptic singularity if $p_g(A)=1$,  because $p_f(A)\le p_g(A)$ and $p_f(A)=0$ implies $p_g(A)=0$.
It is known that for any $n \in \Z_{\ge 1}$ there exists an elliptic singularity with $p_g=n$ (cf. \cite{yau.max}).

We will show that if $A$ is an elliptic singularity, then 
 $U=H^0(\cU)$ and hence 
$A$ is almost Gorenstein. 

For that purpose, we recall a result of K. Konno. 

Let $C_X = C_1 - C_2$, where cycles $C_1\ge 0$ and $C_2\ge 0$ have no common components.
%{\color{red}
Then we have 
%}
 $C_1=\coh$ (see \remref{r:CX}).
The following theorem plays a key role in this section.  %}

\begin{thm}[Konno {\cite[Theorem 4.2]{Konno-coh}}]
\label{t:Ko}
If $A$ is an elliptic singularity, then the
cohomological cycle $C_1$ is linearly equivalent to $-K_X$ on a neighborhood of $\supp (C_1)$.
\end{thm}

\begin{lem}\label{l:gen}
If $A$ is an elliptic singularity, then $h^1(\cO_X)=h^1(\cO_X(K_X+C_X))$.  
Hence $H^0(\alpha_0)$ is surjective and 
we have $U = H^0(\cU)$.  
\end{lem}
\begin{proof}
We have  $\cO_{C_1}(K_X+C_1)\cong \cO_{C_1}$ by \thmref{t:Ko}.
Since $K_X+C_X$ is nef (\remref{r:CX}), 
we have $0 \le (K_X+C_1-C_2)C_1 = -C_2C_1$, and thus $\supp(C_1) \cap \supp(C_2) = \emptyset$.
Therefore, $\cO_{C_1}(K_X+C_X)\cong \cO_{C_1}$. 
By 
\lemref{l:coh}, $H^1(\cO_X(K_X+C_X)) \cong H^1(\cO_{C_1}(K_X+C_X))$.
Hence $h^1(\cO_X(K_X+C_X))= h^1(\cO_{C_1})=p_g(A) = h^1(\cO_X)$.
\end{proof}

\begin{cor}
If $A$ is an elliptic singularity, then 
$\cO_X(K_X+C_X)$ is generated for any resolution $X$  $($cf. \proref{p:KAfree}$)$.
\end{cor}
\begin{proof}
The base points of $\cO_X(K_X+C_X)$ are contained in the support of $\cU$,
and $\cU$ is generated by global sections because it is a coherent sheaf on an affine curve. 
Therefore, $\cO_X(K_X+C_X)$ is also generated since $H^0(\alpha_0) \: H^0(\cO_X(K_X+C_X)) \to H^0(\cU)$ (see Table \ref{tab:KD})
is surjective by \lemref{l:gen}.
\end{proof}

\begin{thm}\label{t:elliptic}
Assume that $A$ is an elliptic singularity. Then 
$A$ is almost Gorenstein.
In fact, if $\m$ is represented by a cycle $Z$ on $X$, 
then we have, for $n\in \Z_{\ge 0}$, 
\[
\ell_A(U/\ol{\m^n} U)=n(K_X+C_X)Z\quad {\text and} \quad \mu(U) = e_0(U) = (K_X+C_X)Z.
\]
%}
\end{thm}
\begin{proof} This follows from Lemma \ref{l:gen} and Theorem \ref{th;Ugg}
%{\color{red}
(and its proof).
%}
\end{proof}

\section{The multiplicity of the module $U$}
\label{s:mult}

%{\color{red}
We use the notation of \S \ref{ss:exactseq} with  $L=K_X+C_X$ and $D=nZ$ ($n\in \Z_{\ge 1}$) and the preceding section, and also use Table \ref{tab:KD}.
We will discuss the case when $U = H^0(\cU)$ does not hold. 
%}

%We preserve the notation of the preceding section and the Table \ref{tab:KD} and 
%discuss the case when $U = H^0(\cU)$ does not hold.  

We fix a general element 
%{\color{red}
$\omega \in K_A$ (cf. \remref{w:choice}), 
%}
and let $U=K_A/ \omega A$.
Assume that $\cO_X(K_X+C_X)$ is generated and that $\m$ is represented by a cycle $Z>0$ on $X$, i.e., $\cO_X(-Z)$ is generated, $\m\cO_X=\cO_X(-Z)$, and $\m=H^0(\cO_X(-Z))$.
Note that we can take  such resolution by \proref{p:KAfree}.
%{\color{red} 
Recall that  $H^0(\cO_X(K_X+C_X)) = K_A$ and $H^0(\cO_X(-nZ))=\ol{\m^n}$.
%}
For $n \in \Z_{\ge 0}$, let 
\[
\eta_n=\ell_A(H^0(\cU(-nZ)) / \ol{\m^n}U).
\]  %}

 In particular, $\eta_0 =0$ if and only if $U = H^0(\cU)$. 

\bigskip 

The following lemma is the most fundamental in our theory.

\begin{lem}\label{l:codim}
%{\color{red} 
 $\eta_n\le \eta_0= p_g(A) - h^1(\cO_X(K_X+C_X))$.  %}
\end{lem}
\begin{proof} From the diagram in Table \ref{tab:KD}, we have the following commutative diagram: 
\begin{equation}\label{eq:mU}
\begin{CD} 
H^0(\cO_{X}(K_X+C_X)) \otimes H^0(\cO_X(-nZ)) @>{\ H^0(\alpha_0) \otimes 1 \ }>>
H^0(\cU) \otimes H^0(\cO_X(-nZ)) \\ 
 @VV{\gamma}V @VV{\delta}V \\
H^0(\cO_X(K_X+C_X-nZ)) @>{\ H^0(\alpha_n) \ }>>  H^0(\cU(-nZ))
\end{CD}
\end{equation}
%{\color{red}
Here, $\gamma$ and $\delta$ are the multiplication maps.  %}
Since the image of $\delta\circ  (H^0(\alpha_0) \otimes 1)$ 
%{\color{red}
coincides with %}
 $\ol{\m^n} U$, 
we have
\[
\ol{\m^n} U \subset \Img H^0(\alpha_n)  \subset H^0(\cU(-nZ)).
\]
On the other hand, taking general elements 
%{\color{red}
$f_1 \in H^0(\cO_X(K_X+C_X))=K_A$ and 
$f_2\in H^0(\cO_X(-nZ)) = \ol{\m^n}$, %}
using \lemref{l:2ndfund}, we have the following exact sequence$:$
\begin{gather}\label{eq:V_n}
\begin{split}
0 \to  V_n:=f_2 K_A + f_1 \ol{\m^n} \to H^0(\cO_X(K_X+C_X-nZ)) \to H^1(\cO_X) 
\\
\to H^1(\cO_X(K_X+C_X)) \oplus H^1(\cO_X(-nZ)) \to 
H^1(\cO_X(K_X+C_X-nZ) ) \to 0.
\end{split}
\end{gather}
Let $\Delta_n:=\ell_A(H^0(\cO_X(K_X+C_X-nZ)) / V_n )$.
Since $V_n \subset \Img \gamma$, 
we obtain that
%{\color{red}
\begin{align*}
\eta_n
&
= \ell_A(\Img H^0(\alpha_n) / \ol{\m^n}U ) + \ell_A(\Coker H^0(\alpha_n) ) 
\\
&
\le \ell_A(H^0(\cO_X(K_X+C_X-nZ)) / \Img \gamma ) + \ell_A(\Coker H^0(\alpha_n) ) 
\\  
& \le \Delta_n + \ell_A(\Coker H^0(\alpha_n)) 
\\
&= [p_g(A) - \left(h^1(\cO_X(K_X+C_X)) + h^1(\cO_X(-nZ)) \right) + h^1(\cO_X(K_X+C_X-nZ) ) ]
\\
& \quad 
+ [h^1(\cO_X(-nZ)) ) - h^1(\cO_X(K_X+C_X-nZ)) ]
\\
&= p_g(A) - h^1(\cO_X(K_X+C_X)) = \ell_A(\Coker H^0(\alpha_0) ) = \eta _0.
\qedhere
\end{align*}
%}
\end{proof}

\begin{rem}\label{r:Delta}
In the situation of the proof of \lemref{l:codim}, 
 if $\Delta_n=0$, then the 
vertical map $\gamma$ in %{\color{red} 
the diagram \eqref{eq:mU} %}
is surjective,  and hence
we have 
$\eta_n =\ell_A(\Coker H^0(\alpha_n))$
%{\color{red}
and $\eta_n=\eta_0$. %}
\end{rem}

\begin{thm}\label{t:HP}
We have $e_0(U)=(K_X+C_X)Z$ and 
%{\color{red}
$\mu(U)=e_0(U)+\eta_1-\eta_0$. %}
In fact, the following formula holds:
\[
\ell_A(U/\ol{\m^n} U)=n(K_X+C_X)Z +\eta_n- \eta_0.
\]
Therefore, $A$ is almost Gorenstein if and only if $\eta_0  = \eta_1$.
\end{thm}
\begin{proof}
We have the following commutative diagram of injections:
\[
\begin{CD}
\ol{\m^n}U @>>> U \\
@VVV @VVV \\
H^0(\cU(-nZ)) @>>> H^0(\cU).
\end{CD}
\]
Thus it follows from \lemref{l:cdU}  that
\begin{align*}
\ell_A(U/\ol{\m^n} U)
&=\eta_n  + h^0(\cU_{nZ})  - \ell_A(H^0(\cU)/U)
\\
&=\eta_n  + n(K_X+C_X)Z - \eta_0.
\end{align*}
Since $\eta_n- \eta_0$ is bounded by \lemref{l:codim}, we obtain $e_0(U)=(K_X+C_X)Z$.
%{\color{red}
Letting $n=1$, we have %}
 $\mu(U)=\ell_A(U/\m U)=(K_X+C_X)Z +\eta_1- \eta_0$.
\end{proof}

%\begin{cor}
%If $U = H^0( \cU )$ holds, then $A$ is almost Gorenstein.
%\end{cor}
%\begin{proof}
%The assumption implies $\eta_0=0$. Hence $\eta_1 = \eta_0$
%{\color{red} 
%by \lemref{l:codim}.  %}
%\end{proof}

The $p_g$-ideals are introduced in \cite{OWYgood}, 
which  have nice properties of integrally closed ideals in rational singularities, and several characterizations 
%{\color{red} 
of $p_g$-ideals are obtained 
in \cite{OWYrees}. %}
For example, we have the following.

%\begin{prop}[cf. {\cite[4.1]{OWYrees}}]
%\label{p:Rees}
%An $\m$-primary integrally closed ideal $I$ of $A$ is a $p_g$-ideal if and only if $I^2=QI$ for a minimal reduction $Q$ of $I$ and %{\color{red} 
%$I$ is normal, namely, %}
%$I^n$ is integrally closed for $n\in \Z_{\ge 1 }$.  
%\end{prop}

%{\color{red} 
\begin{defn-prop}[cf. {\cite[3.10]{OWYgood}, \cite[4.1]{OWYrees}}]
\label{dp:Rees}
Let $I=I_W$ be an $\m$-primary integrally closed ideal of $A$ represented by a cycle $W>0$ on a resolution $Y$ of $\spec(A)$.
Then the following conditions are equivalent:
\begin{enumerate}
\item $h^1(\cO_Y(-W))=p_g(A)$.
\item $\cO_{C_Y}(-W)\cong \cO_{C_Y}$, where $C_Y$ denotes the cohomological cycle on $Y$.
\item $WC_Y=0$.

\item $I^2=QI$ for a minimal reduction $Q$ of $I$ and %{\color{red} 
$I$ is normal, namely, %}
$I^n$ is integrally closed for $n\in \Z_{\ge 1 }$.  
\end{enumerate}
If the conditions above are satisfied, we call $I$ a {\em $p_g$-ideal}.
\end{defn-prop}
For reader's convenience, we outline how (3) implies (2)  as follows. 
There exists $h\in I_W$ such that $\di_X(h)=W+H$, where $H$ has no component of $E$. Then, since $HC_Y=-WC_Y=0$, we have $\supp(H)\cap \supp(C_Y)=\emptyset$.
Therefore, $\cO_{C_Y}(-W)\cong \cO_{C_Y}(H)=\cO_{C_Y}$.

By the definition, if $A$ is a rational singularity, then every $\m$-primary integrally closed ideals are $p_g$-ideals
 (see \exref{ex:nrmpg} for an example
 of non-rational singularities whose  maximal ideals are
$p_g$-ideals).
The following theorem is 
a consequence of \cite[Corollary 11.3]{GTT} (cf. \cite[Proposition 5.4]{PWY}).
We will give another proof using the argument above.

\begin{thm}\label{t:pg}
If $\m$ is a $p_g$-ideal, then $A$ is almost Gorenstein.
\end{thm}
\begin{proof}
By Theorem \ref{t:HP}, it suffices to show that $\eta_0 = \eta_1$. 
It follows from \lemref{l:codim} that $\eta_0 = h^1(\cO_X) - h^1(\cO_X(K_X+C_X)$. 
 By Lemma \ref{K-nZ}, 
we have $\Image(\al_1) = \m U$ and $\eta_1 = h^1(\cO_X(-Z))- h^1(\cO_X(K_X+C_X-Z))$.
Since $h^1(\cO_X(-Z))= p_g(A)$ by our assumption, we have $\eta_0=\eta_1$ 
by Lemma \ref{l:2ndfund}. 
\end{proof}

\begin{quest}
There exists an integer $\bar{e}_1(U)$ such that
\[\ell_A(U/\ol{\m^n} U)=n(K_X+C_X)Z - \bar{e}_1(U)\]
for sufficiently large $n$. Then \lemref{l:codim} 
and \thmref{t:HP}  imply 
 that $\bar{e}_1(U)\ge 0$. 
The proof of  Theorem \ref{th;Ugg} and
\thmref{t:pg} shows that $\bar{e}_1(U)= 0$ if $U = H^0(\cU)$ or 
$\m$ is a $p_g$-ideal. 
%So, we ask the following question: 
Is almost Gorenstein property characterized by $\bar{e}_1(U)=0$?
 We have $\bar{e}_1(U)=\lim_{n\to \infty}(\eta_0-\eta_1)$.
How can we clarify the range of $\defset{\eta_0 - \eta_n}{n\in \Z_{\ge 0}}$ or $\defset{n\in \Z_{\ge 0}}{\bar{e}_1(U) \ne \eta_0 - \eta_n}$?
\end{quest}

\section{Examples of Elliptic singularities}\label{s:Ell}

As is explained in 
%{\color{red} 
\sref{s:Ugg}, 
%}
elliptic singularities plays an important role in this article. 
Since examples of non-Gorenstein elliptic singularities are 
not well known, 
%{\color{red} 
we will explain some concrete examples of graded non-Gorenstein elliptic 
singularities. %}
%we will explain the difference of elliptic and non-elliptic 
%singularities in  some examples. }

For graded singularities (more generally, for a star-shaped singularities (\cite{TW})), elliptic singularities are characterized as follows.

\begin{prop}
[Tomari {\cite[Corollary 3.9]{tomari.max}}]\label{Tomari}
Let $C$ be a nonsingular curve and $D$ a $\Q$-divisor on $C$ with $\deg D>0$.
For every $n\in \Z_{\ge 0}$, we denote by $D^{(n)}$ the  integral part $\fl{nD}$ of $nD$. 
Let $A$ be the localization of  
\[R(C,D):=\bigoplus_{n\ge 0}H^0(C, \cO_C(D^{(n)}))T^n\subset k(C)[T]\] 
with respect to 
the maximal ideal $\m_R=\bigoplus_{n\ge 1}H^0(C, \cO_C(D^{(n)}))T^n$.
Then $A$ is an elliptic singularity if and only if  one of the following conditions holds:
\begin{enumerate}
\item $C$ is an elliptic curve and $\deg D^{(1)}\ge 0$.
\item $C$ is a rational curve and there exist integers 
%{\color{red}
$1\le m_1 < m$ such that 
$m$ is the minimum of positive integers with $\deg D^{(m)}\ge 0$, 
$\deg D^{(m_1)} = -2$ and $\deg D^{(i)} = -1$  
for $1 \le i <m$ with $i\ne m_1$.
\end{enumerate}
Note that the conditions above depend only on the genus $g=g(C)$ and the degrees of $D^{(n)}$ 
$($$n\ge 1$$)$.
\end{prop}

A criterion for the ring $R(C,D)$ to be Gorenstein can be described in terms of $D$ and $K_C$.

\begin{prop}
[{\cite[Corollary 2.9]{Wt}}]
\label{p:Gor}
Assume $D=\sum_{i=1}^m (r_i/q_i)P_i$, where $P_i\in C$ are distinct points, $q_i, r_i\in \Z$, $q_i>0$ and 
%{\color{red}
$\gcd(q_i, r_i)=1$
%}
 for $1\le i \le m$. 
Let $D'=\sum_{i=1}^m (1-1/q_i)P_i$ and let $a(R)$ be the $a$-invariant of $R=R(C,D)$ $($\cite{GW}$)$.
Then $R$ is Gorenstein if and only if $K_C+D'-a(R)D\sim 0$.
$($Note that $a(R)$ is characterized by this property when $R$ is Gorenstein.$)$
\end{prop} 

We give some examples of non-Gorenstein elliptic singularities using \proref{Tomari} and \proref{p:Gor}.
%}

\begin{ex} Let $R(C,D)$ and $A$ be as in Proposition \ref{Tomari} and we put $C= \PP^1$ with rational  
function field $k(C) = k(x)$.  
%{\color{red}%KW0720
We denote by $(a)$ and $(\infty)$, where $a\in k$, the point divisors such that $\divv_C(x-a) = (a) - (\infty)$. %}
\begin{enumerate} 
\item If $D = 1/2 (\infty) + 1/3 (0) - 1/3 [ (1) + (a) ]$, then we can check that $A$ is an elliptic 
singularity and 
not Gorenstein.  The graded ring 
$R( \PP^1, D)$ is generated by $4$ elements  
%{\color{blue}
\[X = \frac{(x-1)(x-a)}{x}T^3 \in R_3, \quad
Y = \frac{((x-1)(x-a))^2}{x} T^6 \in R_6,\]
\[Z=  \frac{((x-1)(x-a))^3}{x^2}T^8 \in R_8, \quad
W = \frac{((x-1)(x-a))^4}{x^3}T^{10} \in R_{10}\]
%}
%$X = (x-1)(x-a)/x T^3 , 
%Y = ((x-1)(x-a))^2/x T^6, Z=  ((x-1)(x-a))^3/x^2 T^8, W = ((x-1)(x-a))^4/x^3 T^{10}$ 
 and we see that 
$R( \PP^1, D) \cong k[X,Y,Z,W]/I_2(M)$, where $I_2(M)$ is the ideal generated
 by the three maximal minors of the matrix 
%{\color{blue} 
\[M=\begin{pmatrix} Y & Z & W \\ Z & W & Y^2 - (a+1)X^2Y + aX^4
\end{pmatrix}.\]  
%}    
 \item If $D = 6/7(\infty) - 3/7(0) - 2/7 (1)$, then we can see that $A$ is an elliptic singularity 
 and $R(\PP^1, D)$ is generated by $6$ elements of 
 degree $6,7,7,9,10,11$, respectively. 
 
We can make examples of elliptic singularities with arbitrary high embedding dimension 
putting $D = a/p (\infty) - b/p(0) - c/p (1)$ such that $p$ is a prime number, $a> b+c$
 and $p> a > p/2$.   
\end{enumerate}
\end{ex} 

%{\color{red}
We give an example of elliptic singularity whose maximal ideal is not $p_g$-ideal; this example shows that the converse of \thmref{t:pg} does not hold.
%}

\begin{ex}\label{ex:ellL=2}
Let $R:=R(C,D)$ and $A$ be as in Proposition 5.1. 
Assume that $C$ is an elliptic curve, $\{P_0, \dots, P_d\}\subset C$ 
%{\color{red}
($d\ge 1$)
%}
 are distinct points, and that
\[
D=d P_0 - \sum_{i=1}^d \frac{1}{2}P_i, \quad D_0 := d P_0  - \sum_{i=1}^d P_i \not \sim 0.
\]
Then $D^{(n)}=nD_0+\sum_{i=1}^d \fl{\frac{n}{2}}P_i$ and  $\deg D^{(n)}=d \fl{\frac{n}{2}}$.
We see that $A$ is an elliptic singularity with $p_g(A)=1$ and not Gorenstein.
In particular, $A$ is almost Gorenstein by \thmref{t:elliptic}.
The Poincar\'e series $ P(R,t) = \sum_{n\ge 0} h^0(\cO_C(D^{(n)})) t^n$ 
of $R$ can be expressed as 
\begin{equation}
P(R,t) = 1 +  \sum_{m\ge 1} dm (t^{2m}+t^{2m+1})  = 
\begin{cases} \frac{1+(d-2) t^2+d t^3+t^4}{(1-t^{2})^2}, & (d\ge 2)\\
              \frac{1+ t^4+t^5}{(1-t^{2})(1-t^3)}  & (d=1).  
\end{cases}
\end{equation}

%{\color{red}
Suppose that $X$ is the minimal resolution.
Then $E$ can be expressed as $E=E_0+\sum_{i=1}^dE_i$, where $E_0\cong C$ with $E_0^2=-d$ and $E_i\cong \PP^1$ with $E_i^2=-2$, $E_0E_i=1$ for $1 \le i \le d$. (cf. \cite[\S 6]{TW}).
Note that $E_0$ is the cohomological cycle.
Let $Z=E+E_0$.
Then $Z$ is anti-nef and $H^0(\cO_X(-Z))=\m$.

Assume that $d=1$.
Then we have $h^0(\cO_C(D^{(2)}))=h^0(\cO_C(D^{(3)}))=1$.
Therefore, $\cO_X(-Z)$ has a base point $p\in E_0\setminus E_1$. 
Let $\phi\: X' \to X$ be the blowing up at $p$ with exceptional divisor $F$ and let $Z'=\phi^*Z+F$.
Then $\cO_{X'}(-Z')$ is generated by $f\in H^0(\cO_C(D^{(2)}))$ and $g\in H^0(\cO_C(D^{(3)}))$, and thus $\m$ is represented by $Z'$.
Since the proper transform $E_0'$ of $E_0$ is the cohomological cycle and $Z'E_0'=0$, it follows from \dpref{dp:Rees} that $\m$ is a $p_g$-ideal. 
This case will be treated in \exref{ex:d=1} as a determinantal singularity.

Next assume that $d\ge 2$. 
Since $\cO_C(D^{(2)})$ is generated, $\m$ is represented by the cycle $Z$.
In fact,  if $f,g\in H^0(\cO_C(D^{(2)}))$ are general elements, then $Q=(f,g)A$ is a minimal reduction for $\m$, and they generate $\cO_X(-Z)$.
Since $ZE_0=-d$, $\m$ is not $p_g$-ideal by \dpref{dp:Rees}.
Moreover, by \cite[3.14, 3.16]{ORWY}, we have $\ell_A(\ol{\m^2}/Q\m)=1$ and the following:
\begin{enumerate}
\item If $d=2$, then $\ol{\m^2}\ne \m^2=Q\m$.
\item If $d\ge 3$, then $\ol{\m^2}= \m^2 \ne Q\m$, and $\m$ is normal.
\end{enumerate}
%}
\end{ex}
%}

\section{Examples of cone singularities}
\label{s:example}

In this section, we will study the almost Gorenstein property 
for \lq\lq cone  singularities",
which are defined by an ample divisor $D$ on a smooth curve $C$ of genus $g$. 
If $g=0$ (resp. $g=1$), then the singularities are rational (resp. elliptic).
{\bf So we will always assume that $g\ge 2$ in this section. }

We will use the notation from the preceding sections and set up the following:

%\begin{ass}\label{Cone;SetUp}
\begin{not-prop}\label{Not6.1} 
\begin{enumerate}
\item  Let $C$ be a curve of genus $g\ge 2$ and $D$ a divisor on $C$ with $\deg D >0$.
Then $\cO_C(K_C)$ is generated. 

\item Let $R=R(C,D):=\bigoplus_{n\ge 0}H^0(C, \cO_C(nD))$ and 
$\m_R=\bigoplus_{n\ge 1}H^0(C, \cO_C(nD))$. \par
The $a$-invariant $a(R)$ is defined by $a(R) = \max\{n \;|\; [H^2_{\m_R}(R)]_n \ne 0\;\}$
(\cite{GW}). Also, by \cite[Theorem 2.8]{Wt}, we have 
\[
K_R=\bigoplus_{n\ge -a(R)} H^0(C, \cO_C(K_C+nD)).
\]  

%}

\item 
We define $A$ as the localization of $R$ with respect to the maximal ideal $\m_R$.
Then the exceptional set of the minimal resolution of $A$ is isomorphic to $C$, and we have $p_f(A)=g$, $p_g(A)= \sum_{n= 0}^{a(R)}h^1(C, \cO_C(nD))$  (cf. \cite[Theorem 3.3]{Wt}, or \sref{s:det}) and 
\begin{equation}\label{eq:type}
\type(A) = \dim_k K_R / \m_R K_R.
\end{equation}

\item (\cite{Wt})  Let $X$ be the minimal resolution of 
 $\spec (R)$
 with unique exceptional curve 
$E\cong C$. Then we can write  
\[X = \Spec_C ( \bigoplus_{n\ge 0} \cO_C(nD) T^n)) \; {\rm and} \; 
\cO_X ( rE) \cong  ( \bigoplus_{n\ge -r } \cO_C(nD) T^n)^{\widetilde{}}.\]
 Namely, $X$ is an $\bbA^1$-bundle over $C$.
Hence there are exact sequences 
\[
0\to H^0(\cO_{X}(-(n+1)E)) \to H^0(\cO_{X}(-nE)) 
\to H^0(\cO_{E}(-nE)) \to 0
\]

for every $n \in \Z$ and we have $\cO_{E}(-nE) \cong \cO_C( nD)$.

\item Since $\cO_X( K_X + E)\otimes_{\cO_X}\cO_E \cong \cO_C( K_C)$, 
%{\color{red} 
we have 
\[
H^0(X, \cO_X(K_X + rE) ) \cong \bigoplus_{n\ge -r +1}H^0( C, \cO_C(K_C + nD)) \quad (r\in \Z).
\] 
In particular, we have $C_X = (a(R)+1) E$.
% and $\cO_X( K_X + (a+1)E)$ is generated.        
%}
\end{enumerate}
\end{not-prop}

\begin{rem}
(1) If $\deg D > \deg K_C$, then $H^1(C, \cO_C(nD))=0$ for $n\ge 1$ and 
%{\color{red} 
$a(R)=0$. %$K_R=\bigoplus_{n\ge 0}H^0(C, \cO_C(K_C+nD))$.

(2) If $U = H^0(X, \cU)$ (cf. \sref{s:Ugg}), then $R$ is Gorenstein.
In fact, if the equality holds, then we have the isomorphism of graded modules
\begin{gather*}
H^1(\cO_X) = \bigoplus_{n=0}^{a} H^1(C, \cO_C(nD)) \\
\to H^1(\cO_X(K_X+(a+1)E)) = \bigoplus_{n\ge -a}^0 H^1(C, \cO_C(K_C+nD))
\end{gather*}
of degree $-a$, where $a=a(R)$.
Since this implies the isomorphisms $H^i(C, \cO_C)\cong H^i(C, \cO_C(K_C-aD))$ for $i=0,1$, we have $K_C\sim aD$. 
%}
\end{rem}

The following theorem of Stanley is very important in our theory. 

\begin{thm}[Stanley's Theorem; \cite{St-recip}]\label{Stan-Thm} 
Let $S = \bigoplus_{n\ge 0} S_n$ be a Cohen-Macaulay graded ring of dimension $d$ 
with $S_0 =k$, a field and let 
%{\color{red}
the Poincar\'e series of $S$ be 
%}
\[P(S, t) := \sum_{n\ge 0}\dim_k S_n t^n = \dfrac{h_0 + h_1t + \cdots + h_st^s}
{\prod_{j = 1}^d ( 1 - t^{b_j})}.\]
Then the Poincar\'e series of the canonical module $K_S$ is given by 
\[
P(K_S, t) = (-1)^d P(S, t^{-1}) = 
%{\color{red}
\frac{h_st^{B-s} + \cdots + h_1t^{B-1}+ h_0 t^B}{\prod_{j = 1}^d ( 1 - t^{b_j})}, 
%}
\]
%{\color{red}
where $B = \sum_{j=1}^d b_j$.     
\end{thm}

%{\color{red}
\begin{conv}
Let $S = \bigoplus_{n\ge 0} S_n$ 
be a graded ring with $S_0 =k$.
Then we say that $S$ is almost Gorenstein if the localization of $S$ with respect to the maximal ideal $\m_S=\bigoplus_{n\ge 1} S_n$ is an almost Gorenstein local ring.

For example, if $S$ is almost Gorenstein graded ring in the sense of \cite[Definition 1.5]{GTT}, then $S$ is almost Gorenstein because the multiplicity of a graded $S$-module coincides with that of the localization with respect to $\m_S$.
\end{conv}
%}
We give examples of cone singularities which are 
almost Gorenstein and not Gorenstein. 

\subsection{Level Cone Singularities and Standard Graded Sings.}
\phantom{AAAAAAA}\par 

The concept of {\bf level rings} are usually used for standard graded rings. 
Here we use this word in 
%{\color{red} 
the following sense.

\begin{defn}\label{level} 
We say that $R= R(C,D)$ is {\bf level} if $\cO_C(D)$ is 
generated and $K_R$ is generated by $[K_R]_{-a(R)}$ as an $R$-module. 
The condition \lq\lq $\cO_C(D)$ is 
generated" is equivalent to the condition that
we can take $f_1, f_2 \in R_1$ so that $Q = (f_1, f_2)R$ is an $\m_R$-primary
 ideal of $R$. If $R$ is Gorenstein and if $\cO_C(D)$ is 
generated, then $R$ is trivially level.
\par
If $R$ is level, then $\type (R) = \dim_k [K_R]_{-a(R)} = h^0(\cO_C(K_C - a(R)D))$. 
\end{defn} 
%}

%%{\color{red} 
%\begin{defn}[{\cite[Definition 1.5]{GTT}}]
%Let $S = \bigoplus_{n\ge 0} S_n$ 
%be a Cohen-Macaulay graded ring of dimension $d$ with $S_0 =k$.
%Then $S$ is said to be almost Gorenstein if there exists an exact sequence 
%\[
%0\to S \to K_S(-a(S)) \to U \to 0
%\]
%of graded $S$-modules with $\mu_S(U) = e_{S}(U)$, where $\mu_S$ denotes the number of elements in a minimal system of generators of the $S$-module $U$ and $e_{S}$ the multiplicity with respect to the homogeneous maximal ideal $\m_S \subset S$.
%
%Note that if $S$ is almost Gorenstein , then so is the localization of $S$ with respect to $\m_S$ because the multiplicity of a graded $S$-module coincides with that of the localization with respect to $\m_S$.
%\end{defn}
%%}

We refer a result of A. Higashitani giving a sufficient condition for a standard 
graded ring to be almost Gorenstein.

\begin{thm}[\cite{Hig} Theorem 3.1]
\label{Higashi}
 Let $k$ be a field and let $S = \bigoplus_{n\ge 0} S_n$ be a 
Cohen-Macaulay graded ring of dimension $d$ satisfying the condition $S= k[S_1]$. 
Let the Poincar\'e series of $S$ be 
\[ P(S,t) :=\sum_{n\ge 0} \dim_k S_n t^n = \frac{h_0 + h_1t + \cdots + h_st^s}{(1-t)^d}.\]
If $h_i = h_{s-i}$ for $i =0, 1, \ldots , [s/2] -1$, then $S$ is almost Gorenstein.
\end{thm}

%{\color{red}
\begin{ex}
[cf. {\cite[Lemma 10.2]{GTT}}]
\label{level=nAG}
If $R$ is level and  not Gorenstein (and $g\ge 2$), then $R$ is {\bf not} almost Gorenstein. 
In fact, using Stanley's Theorem, we can show that $e_0(U) > \mu(U) = \type(R) -1$, where $U = K_R/ w R$  for some $w\in [K_R]_{-a(R)} \setminus\{0\}$. 
\end{ex}
%}

\begin{rem} If $R(C,D)$ is an elliptic singularity, then $R$ is a \lq\lq simple elliptic  singularity" and is Gorenstein.  
\end{rem}

The following cases are typical examples of level rings. 

%{\color{red}
\begin{prop}\label{ex:notAG}
$R=R(C,D)$
 is level and hence not almost Gorenstein 
 in the following cases.

\begin{enumerate}
\item [(a)]  $\deg D > 2 \deg (K_C)$.
\item [(b)]
Let $R'$ be a standard graded Gorenstein ring with $a(R')>0$
and  $R(C,D) = (R')^{(r)}$ is an $r$-th Veronese subring with $r \ge 2$ and $r$ does not 
divide $a(R')$.
\end{enumerate}  
\end{prop}
\begin{proof} 
The condition (a) implies that $\cO_C(D)$ is 
generated. Hence by \lemref{l:generation}, $R$ is level in case (a). 
Next we consider the case (b).
 Since $K_{R'}$ is generated by an element $w\in (K_{R'})_{-a(R')}$ and since $R'$ is standard graded, $K_R \cong (K_{R'})^{(r)}$ \cite[(3.1.3)]{GW}
is also generated by elements of 
%{\color{red}
$(K_{R'})_{m}$, where $m=-r \fl{a(R')/r}$. 
%is the largest integer smaller than $a(R)/r$.  
%}
\end{proof}
%}

%We ask \lq\lq If $R=R(C,D)$ is a {\bf non-Gorenstein} standard graded ring with $g(C)\ge 2$, 
%then $R$ is not almost Gorenstein ?" but we can make such examples. We will see later 
%that if $g=2$, then we do not have such examples. 
%So, we think that  a standard graded ring $R$ which is not Gorenstein and almost Gorenstein are limited.  

\begin{ex}[cf. {\cite[Corollary 7.10]{GTT}}]
 Let $S$ be a polynomial ring with $4$ variables, generated by elements of degree $1$ 
over $k$.
%{\color{red}
 Let 
\[M=\begin{pmatrix} l_1 & l_2 & l_3 \\ g_1 & g_2 & g_3
\end{pmatrix},\]
where $l_1, l_2, l_3 \in S$ are linear forms and  $g_1,  g_2,  g_3\in S$ are homogeneous polynomials 
of degree $s \ge 2$, and let $I_2(M)$ denote the ideal generated by maximal minors of $M$.
Let $R = S/I_2(M)$.
%}
 We assume that $R$ is normal. Then $R$ is almost Gorenstein. \par
Actually, the minimal free resolution of $R$ over $S$ is 
\[
0 \to S(-2s-1)\oplus S(-s-2)     \to S( -s-1)^{\oplus 3} \to S \to R \to 0.
\] 
%{\color{red}
Therefore,
%} 
we have 
\begin{align*}
P(R, t) &= \dfrac{1- 3t^{s+1} + t^{s+2} + t^{2s+1}}{(1-t)^4}\\
&= \frac{1+ 2t + \cdots  + (s+1)t^s + (s-1) t^{s+1} +(s-2) t^{s+2}+ \cdots +t^{2s-1}}{(1-t)^2}.
\end{align*}
Hence $R$ is almost Gorenstein by \thmref{Higashi}.\par
%{\color{red} 
Note that in this example, $g(C) = \dim_k [K_R]_0$ can be computed by \thmref{Stan-Thm}.
%}

\end{ex}

\begin{quest} For what $g$, does there exist $R(C,D)$, which is standard graded, 
not Gorenstein, and almost Gorenstein with $g(C) = g$? We will show later that 
there is no such $R$ if $g=2$.
We think that  a standard graded ring $R$ with $g\ge 2$,  which is not Gorenstein and almost Gorenstein are very few. 
\end{quest}

\subsection{The case $D = K_C + P$} 
\phantom{AAAAAAA}\par 

Next, we give an example of $R(C,D)$ which is almost Gorenstein if $C$ is hyperelliptic, 
and not almost Gorenstein if $C$ is not hyperelliptic. 
This example shows that the almost Gorensteinness cannot be determined by the pair $(g, \deg D)$, and depends on the structure of $C$. %}

\begin{thm}\label{p:K+p}
Assume that $P\in C$ is a point and we put $D=K_C+P$.
Then $A$ is not Gorenstein,  $p_g(A)=g$, and $e_0(U)=2g-2$.
 \begin{enumerate}
\item If $C$ is not hyperelliptic, then  
%{\color{red}
$\mu(U)= g$; %}
hence $A$ is not almost Gorenstein.
\item If $C$ is hyperelliptic, then  $\mu(U)= 2g- 2$; hence $A$ is almost Gorenstein.
\end{enumerate} 
\end{thm}
\begin{proof}
Let $X_0 \to \spec (A)$ be the minimal resolution with exceptional set $E_0\cong C$.
Since $\deg D \ge 2g -1$, we have  $h^1(\cO_C(nD))=0$ for $n\ge 1$, and $p_g(A)=h^1(\cO_C)=g$. 
%{\color{red}
We see that $A$ is not Gorenstein by \cite[Corollary 3.6]{Konno-coh} or 
\proref{p:Gor}. %}
For $n\ge 1$,
%{\color{red}
since $\cO_{X_0}(- E_0) \otimes_{\cO_{X_0}} \cO_{E_0} \cong \cO_{E_0}(D)$,
%}
 we also have $H^0(\cO_{{E_0}}(K_{X_0}+(n+1){E_0}))\cong H^0(\cO_{C}(K_C-nD))=0$.
From the exact sequence
\[
0 \to \cO_{X_0}(K_{X_0}+n{E_0}) \to \cO_{X_0}(K_{X_0}+(n+1){E_0}) \to 
\cO_{E_0}(K_{X_0}+(n+1){E_0}) \to 0,
\]
we obtain $H^0(\cO_{{X_0}}(K_{{X_0}}+{E_0}))=H^0(\cO_{{X_0}}(K_{X_0}+(n+1){E_0}))$
for $n\ge 1$, 
and hence $C_{{X_0}}={E_0}$. 
From the Grauert-Riemenschneider vanishing theorem (cf. \cite[1.5]{gi}) and the exact sequence 
\[
0 \to \cO_{{X_0}}(K_{{X_0}}) \to \cO_{{X_0}}(K_{{X_0}}+{E_0}) 
\to \cO_{{E_0}}(K_{{E_0}}) \to 0,
\]
we have the surjection 
$H^0(\cO_{{X_0}}(K_{{X_0}}+{E_0}) ) \to H^0(\cO_{{E_0}}(K_{{E_0}}))$.
Thus $\cO_{{X_0}}(K_{{X_0}}+C_{{X_0}})$ is generated.
Since $h^1(\cO_C(K_C)) =1$, $h^1(\cO_C(D))=0$, 
from the exact sequence 
\[0\to \cO_C(K_C) \to \cO_C(D) \to \cO_{P} \to 0, \]
we have $H^0(\cO_C(K_C))=H^0(\cO_C(D))$.
Thus $\{P\}$ is the base locus of $\cO_C(D)$. 
From  
the exact sequence (see  Notation-Proposition 
%{\color{red}
\ref{Not6.1} (4))  
%}
\[
0\to H^0(\cO_{X_0}(-2E_0)) \to H^0(\cO_{X_0}(-E_0)) 
\to H^0(\cO_{E_0}(D)) \to 0,
\]
we have that  $\m\cO_{X_0}=I_P\cO_{X_0}( -E_0)$, where $I_P$ is the ideal sheaf of the point $P$; in other words, $E_0$ is the maximal ideal cycle on $X_0$ (cf. \cite[\S 4]{yau.max}) and $\cO_{X_0}(-E_0)$ has a single base point $P$ of multiplicity one.
Let $\phi\: X\to X_0$ be the blowing up at $P$ and $E_1$ the exceptional set of $\phi$.
We denote by $E_0$ again  the proper transform of $E_0$ on $X$.
Then the maximal ideal $\m\subset A$ is represented by $Z:=\phi^*E_0 + E_1 = E_0+2E_1$.
%{\color{red}
Since $\cO_{X_0}(K_{X_0}+E_0)$ is generated and $C_{X_0}=E_0$, we have  
%$p_g(A)=g$, the cohomological cycle on $X$ is $E_0$. Thus 
$\phi^*(K_{X_0}+E_0)=K_X+E_0 = K_X +C_X$ and $\cO_{X}(K_{X}+C_X)$ is generated.
%}
%, because  $\cO_{X_0}(K_{X_0}+E_0)$ is generated.  
By \thmref{t:HP}, 
since $K_XE_1=-1$, 
\[
e_0(U)=(K_X+E_0)(E_0+2E_1)=2g-2 + (K_X+E_0)(2E_1) = 2g-2.
\]

Now we need a lemma to compute $\mu(U)$.

\begin{lem}\label{l:generation} 
Let $D$ be a divisor on $C$ with $\deg (D)>0$.
For $n \ge 1$, 
%}
consider the multiplication map
\[
\gamma_n : H^0(\cO_C(K_C)) \otimes H^0( \cO_C(nD)) \to H^0(\cO_C(K_C+nD)).
\]
Then  we have the following.
\begin{enumerate}
\item 
If $h^1(\cO_C(nD - K_C)) = h^0(\cO_C(2 K_C - nD)) =0$, 
then $\gamma_n$ is surjective. 
In particular, if $\deg nD> 2 \deg (K_C)$, then $\gamma_n$ is surjective.

\item Let $D=K_C+P$, where $P$ is a point of $C$. 
Then $\gamma_n$ is surjective for $n\ge 2$. 
 \begin{enumerate}
\item If $C$ is not hyperelliptic, then  $\dim_k \Coker \gamma_1 = 1$.
\item If $C$ is hyperelliptic, then  $\dim_k \Coker \gamma_1 = g-1$.
\end{enumerate} 
\end{enumerate}
\end{lem}
\begin{proof}
Let $a,b\in H^0(\cO_C(K_C))$ be general elements that generate $\cO_C(K_C)$
 at each point of $C$. Then,  
we have the following exact sequence (cf. \lemref{l:2ndfund}):
\begin{equation}\label{eq:seqC}
0 \to \cO_C(-K_C+nD) \xrightarrow {(a,b)}
 \cO_C(nD)^{\oplus 2} \xrightarrow{\genfrac{(}{)}{0pt}{}{-b}{a}}
\cO_C(K_C+nD) \to 0.
\end{equation}
This implies (1)  since $H^1(\cO_C(-K_C+nD))=0$ by our assumption.
%for $n \ge 1$. 
\par
Assume that $D=K_C+P$. 
Then, if $n\ge 2$, we have $-K_C+nD > K_C$, and hence $H^1(\cO_C(-K_C+nD))=0$.
Since $\cO_C(K_C)$ is generated,  
$h^1(\cO_C(- K_C +D))= h^0(\cO_C(K_C-P)) = g-1$. 
Therefore, the exact sequence \eqref{eq:seqC} implies that $\gamma_n$ is surjective for $n\ge 2$ and 
  $\dim_k \Coker \gamma_1 \le h^1(\cO_C(-K_C+D)) = g-1$. 
 %{\color{red}
For the following argument, recall that %}
$H^0(\cO_C(K_C))=H^0(\cO_C(D))$.

If $C$ is not hyperelliptic, then by Max Noether's theorem (cf. \cite[(2.10)]{Saint-Donat}), the multiplication map $H^0(\cO_C(K_C)) \otimes H^0( \cO_C(K_C)) \to H^0(\cO_C(2K_C))$ is surjective, and since $H^0( \cO_C(D))=H^0( \cO_C(K_C))$ and 
$h^0(\cO_C(K_C+D)) = h^0(\cO_C(2K_C))+1$, 
\[
\dim_k \Coker \gamma_1 = \ell_A( H^0(\cO_C(2K_C+P)) / H^0(\cO_C(2K_C))) =1.
\]

Next assume that $C$ is hyperelliptic.
Then there exists a double covering $\pi\: C \to \PP^1$. If we put $F = \pi^{*}(Q)$ for some point 
$Q \in\PP^1$,  then $(g-1) F \sim K_C$ and we have 
\[
R(C, F) =\bigoplus_{n\ge 0} H^0(\cO_C(nF)) \cong k[X, Y, Z]/(Z^2 - f_{2g+2}(X,Y))
\]
for some homogeneous polynomial $f_{2g+2}$ of degree $2g+2$, if $p =\chara(k) \ne 2$
(if $p=2$, we need slightly different expression but the argument is essentially the same).
Since 
 $ H^0( \cO_C(D)) = H^0( \cO_C(K_C))  \cong R(C,F)_{g-1}\cong k[X,Y]_{g-1}$, 
we have  $\Img \gamma_1
\cong  k[X,Y]_{2(g-1)}$, 
and thus  
$\dim_k (\Img \gamma_1) = 2g -1$.
Since 
%{\color{red}
$h^0(\cO_C(2K_C+P))=\deg(2K_C+P) + 1 - g = 3g-2$, 
 we obtain that
\[
\dim_k \Coker \gamma_1 = h^0(\cO_C(2K_C+P)) - \dim_k \Img \gamma_1 =  g-1.
\qedhere
\]
\end{proof}

Now we apply \lemref{l:generation} (2)
to complete the proof of Theorem \ref{p:K+p}.  
By \eqref{eq:type} and \lemref{l:generation},  we have 
\begin{align*}
\mu(U)  =\type(A)-1 & = 
h^0(\cO_C(K_C))+ \dim_k \Coker \gamma_1 -1 \\
&= g - 1 +\dim_k \Coker \gamma_1 \\
& =\begin{cases}
g &  \text{ if $C$ is not hyperelliptic, }\\
2g-2 & \text{ if  $C$ is  hyperelliptic.}
\end{cases}
\qedhere
\end{align*}
%}
\end{proof}

\subsection{The Point Divisors and Weierstrass Semigroups}
\phantom{AAAAAAA}\par 

Next we treat the Case when $D = P \in C$ is a point. 
%{\color{red}
The ring $R(C,P)$ 
%}
 has very strong connection with the Weierstrass semigroup 
%{\color{red}
$H_{C,P}$. We
%}
 will show that $R(C,P)$ is almost Gorenstein 
if and only if $H_{C,P}$ is almost symmetric, or, the semigroup ring $k[H_{C,P}]$ is almost 
Gorenstein. 
First let us review the \lq\lq Weierstrass semigroup"
\[
H_{C,P} = \defset{n\in \Z_{\ge 0}}{ h^0(\cO_C(nP)) > h^0(\cO_C((n-1)P)) }.
\]
The finite set $\Z_{\ge 0}\setminus H_{C,P}$ is called the \lq\lq Weierstrass gap sequence"
 and the cardinality $\sharp(\Z_{\ge 0}\setminus H_{C,P})$  
coincides with $g=g(C)$. 

We have the following relation between $R(C,P)$ and $H_{C,P}$. 

\begin{prop}[{\cite[Proposition (5.2.4)]{GW}}]\label{R/TR}
If $R=R(C,P)= \bigoplus_{n\ge 0}H^0(\cO_C(nP))T^n$ then for $T\in R_1$, 
 $R/TR\cong k[H_{C,P}]$.  
\end{prop}

In general,  we say that a subset $H\subset \Z_{\ge 0}$ is a {\it numerical semigroup} if  
$H$ is closed under addition, $0\in H$, and $\Z_{\ge 0}\setminus H$ is a finite set. 
Assume that $H$ is a numerical semigroup.
We define  
$g(H) = \sharp (\Z_{\ge 0}\setminus H)$,
$\Fr(H) =  \max\{ n \in \Z\;|\; n \not\in H\}$, and 
$\type(H)=\sharp \defset{n\in \Z\setminus H}{n+H\cap \Z_{>0} \subset H}$.   
Then we have always $2 g(H) \ge \Fr(H) + \type(H)$ and we say that $H$ is an
 {\it almost symmetric 
semigroup} if $2 g(H) = \Fr(H) + \type(H)$ holds. 

\begin{prop}[\cite{BF}]
\label{p:BF}
$k[H] = k[ t^h\;|\; h\in H]$ is almost Gorenstein if and only if $H$ is almost symmetric. 
\end{prop}

\begin{prop}\label{p:mpg}
Assume that  $D = P \in C$ is a point.
Then $\m$ is a $p_g$-ideal if and only if $H_{C,P}=\{0, g+1, g+2, \dots\}$.
%Hence $A$ is almost Gorenstein if $H_{C,P}=\{0, g+1, g+2, \dots\}$. 
Note that this condition is satisfied for general point $P\in C$.
\end{prop}
%}

\begin{proof}
The ``if part'' 
follows from \cite[Example 5.5]{PWY}.
Assume that $\m$ is a $p_g$-ideal. Then $A$ has minimal multiplicity and so is $k[H_{C,P}]$.
Let $e=\min (H_{C,P} \cap \Z_{>0})$ and suppose that $e < g+1$.
Then there exists $n\in H_{C,P}$ which is a member of minimal generators of $H_{C,P}$ and $n > 2e$, since $H_{C,P}$ should be generated by $e$ elements.
Then we can take $f\in R_n \setminus \m_R^2$, corresponding to $n\in H_{C,P}$,  such that its image in $A$ is integral over $\m^2$; it contradicts that $\m$ is a $p_g$-ideal (see \dpref{dp:Rees}).
\end{proof}

%{\color{red} 
 From \proref{p:BF}, \proref{p:mpg}, and \thmref{t:pg}, we have the following.

\begin{cor}
Assume that  $D = P \in C$ is a point.
Then  $A$ is almost Gorenstein if $H_{C,P}=\{0, g+1, g+2, \dots\}$. 
\end{cor}
%}

%{\color{red} 
\begin{thm}
 [{\cite[Theorem 3.7]{GTT}}]
\label{t:GTT3.7}
Let $f\in \m$ be an $A$-regular element.
If $A/fA$ is almost Gorenstein, then $A$ is almost Gorenstein. 
The converse holds if $f$ is superficial for the $A$-module $U$ $($see \sref{s:Ugg}$)$ with respect to $\m$.
\end{thm}

\begin{thm} 
\label{t:Hpg}
Assume that $R=R(C,P) = \bigoplus_{n\ge 0} H^0( C, \cO_C(nP)) T^n$.
Then the following conditions are  are equivalent$:$
%, then $T\in R(C,P)$ is a superficial element for $U$ and hence if $A$ is almost Gorenstein, so is $k[H_{C,P}]$.}
\begin{enumerate}
\item $A$ is almost Gorenstein. 
\item $k[H_{C,P}]$ is almost Gorenstein. 
\item $H_{C,P}$ is almost symmetric. 
\end{enumerate}
\end{thm}
%}
\begin{proof} 
%{\color{red} 
By  \proref{p:BF} and \thmref{t:GTT3.7}, it suffices to prove that (1) implies (2). 
%}
Assume that $X$ is the minimal resolution of $\Spec(R(C,P))$. 
We have seen in Notation-Definition \ref{Not6.1} that $X=\Spec_C( \bigoplus_{n\ge 0}\cO_X(nP))$
%}
 and $C_X = (a+1)E$, where $a = a(R)$. 
%{\color{red}
Let $H=H_{C,P}$.
Now, since $a(k[H]) = \Fr(H)$ (cf. \cite[2.1.9]{GW}) and 
$R/TR\cong k[H]$,
%}
 we have $a(R) = a(k[H]) -1$. 
Assume that $A$ is almost Gorenstein.
Since by Remark \ref{w:choice}, we can take $\omega\in K_A$ to be the image of a homogeneous element $w = f T^{-F(H)}$, where 
%{\color{red}
$f \in H^0(\cO_C(K_C - F(H) P))$. 
%}
Then $K_R/ wR$ is a graded $R$-module of dimension $1$ and 
%{\color{red}
the assertion follows %it suffices to show the following 
from Lemma \ref{min:deg}. 
%}
Actually $e_0( K_R/ wR) = \dim_k (K_R/ wR)/ T (K_R/ wR)$ by Lemma \ref{min:deg}
and  since $R/TR\cong k[H_{C,P}]$ we have  $\type(R)= \type (k[H])$.
\end{proof}

\begin{lem}\label{min:deg}  
Let $S= \bigoplus_{n\ge 0}S_n$ be a graded ring and let
$M$ be a finitely generated Cohen-Macaulay graded $S$-module of dimension $1$.
If $x\in S_b$ is an $M$-regular element and $S_i = 0$ for $0< i< b$, then 
we have $e_0(M) = \dim_k M/xM$.
\end{lem}
\begin{proof} Replacing $S$ by $S/\Ann_S(M)$, we may assume that $\dim S =1$. 
Since $S/xS$ is an Artinian ring, $\m_S^n \subset xS$ for some $n\ge 2$ and comparing the 
degrees, we can assert $\m_S^n = x \m_S^{n-1}$, which shows that $(x)$ is a minimal reduction 
of $\m_S$. 
\end{proof}

%Now, we have shown that $(K_R/ wR)/ T (K_R/ wR)\cong K_{k[H]}/ z k[H]$, where $z \in (K_{k[H]})_{-F(H)}$.
%Since $\type(R)= \type (k[H])$, we have proved the following. 

%\begin{thm}\label{t:AGHCP}
%{\color{red}
%$A$  %$R(C,P)$ 
%is almost Gorenstein if and only if $k[H_{C,P}]$ is almost Gorenstein 
%(or $H_{C,P}$ is almost symmetric).
%}
%\end{thm}

\subsection{
%{\color{red}
The case $g=2$
%}
} 
\phantom{AAAAAAA}\par 

In this subsection, we determine almost Gorenstein property of 
all possible $R= R(C,D)$, when $C$ has genus $2$. 

\begin{thm} Let $R=R(C,D)$ with $g(C)=2$. Then we have the following;
\begin{enumerate}  
\item If $\deg(D) \le 2$, then $R$ is almost Gorenstein.
\item If $\deg(D) = 3$, $R(C,D)$ is almost Gorenstein if and only if $\cO_C(D)$ is 
{\em not} generated.
\item If $\deg(D) \ge 4$ and $D \not\sim 2K_C$, then $R$ is {\em not} 
almost Gorenstein. 
If $D\sim 2K_C$, then $R$ is almost Gorenstein.
\end{enumerate}
\end{thm}
\begin{proof}
%{\color{red}
We write $h^i(L) = h^i(\cO_C(L))$ for a divisor $L$ on $C$.
 By Riemann-Roch Theorem, we have 
\[ 
h^0(L ) = \deg (L) + h^1(L) - 1 \ge \deg (L) -1,
\]
and $h^0(L) \le \deg(L)$ if $\deg(L)>0$.
We also have that
 $h^1(L) = 0$ if $\deg(L) \ge 3$, or $\deg(L)=2$ and $L \not\sim K_C$, or 
$\deg(L)=1$ and $h^0(L)=0$. 
%}
We determine the almost Gorenstein property of $R$ according to $\deg(D)$.    
\begin{enumerate}
\item If $\deg(D)=1$ and $h^0(D) = 1$, then $D = P$ for some point $P\in C$. 
Since the possibility of $H_{C,P}$ with $g=2$ is either 
$\langle 2, 3\rangle $ or $\langle3,4,5\rangle$, 
$R$ is almost Gorenstein by \thmref{p:mpg}. 

\item 
Assume $\deg(D)=1$ and $h^0(D) = 0$.
Then $1 \le h^0(2D) \le 2$.

\begin{enumerate}
\item If $h^0(2D)=2$, then $2D\sim K_C$ and $R$ is Gorenstein. 
\item If  $h^0(2D) = 1$, then  $a(R)=0$ and $h^0(nD) = n-1$ for all $n\ge 1$ and thus 
\[P(R,t)= 1+ t^2+2t^3+ \ldots 
= \frac{1 + t^3+2t^4+2t^5}{(1-t^2)(1-t^3)}.
\] 
Let $w \in (K_R)_0$ be a nonzero element.
By \thmref{Stan-Thm}, we have 
\begin{align*}
P(K_R,t) &=\frac{2 + 2t + t^2 + t^5}{(1-t^2)(1-t^3)}, \\
 P(K_R/wR, t) &= \frac{1 + 2 t + t^2 - t^3 - 2 t^4 - t^5}{(1 - t^2)(1 - t^3)}
=\frac{1 + 2 t + t^2}{1 - t^2}.
\end{align*}
Hence $e_0( K_R/wR) = 4$ and we can check that $K_R$ is minimally generated by elements of degrees $0$, $1$, and one element of degree $2$,   and thus $\type(R)=5$. 
Since  $e_0( K_R/wR) = \type(R) - 1$, $R$ is almost Gorenstein. 

%Since $\cO_C(3D)$ is generated, for some $f\in R_2$ and $g\in R_3$, $(f,g)$ forms an 
%$R$-regular sequence and $P(R,t)= \frac{(1+t^3+2t^4+2t^5)}{(1-t^2)(1-t^3)}$.
%Since $\dim_k [K_R]_0 =2, \dim_k [K_R]_{-1} = h^0( K_C-D) \cong h^1(D) =0,  
%\dim_k [K_R]_1 = h^0(K_C +D) =3$, $K_R$ is generated by degrees $0$ and $1$  
% and $\type(R)=5$.  
%We have $a(R)=0$ and for some $\ne 0, w\in [K_R]_0$, $P(K_R/wR, t) = (1+2t+t^2)/(1-t^2)$.
%We can also say that since the least degree of non zero homogeneous element of 
%$\m_R$ is $2$, this shows $e_0(K_R/wR)=4= \type(R)-1$ and $R$ is almost Gorenstein.   
\end{enumerate}

\item Assume $\deg D=2$. If $h^0(D)=2$, then $D\sim K_C$ and $R$ is Gorenstein. 
We assume that $h^0(D)=1$. 
Then $a(R)=0$ and $h^0(nD) = 2n-1$ for all $n\ge 1$.
Hence we have 
\[
P(R,t) =\frac{1+t^2+ 2t^3}{(1-t)(1-t^2)}, 
\quad P(K_R, t)  =\frac{2+t+t^3}{(1-t)(1-t^2)}.
\]
We see that $\type(R)=3$ since $K_R$ is generated by two elements of degree $0$ and one element of degree $1$,  and that $e_0(K_R/wR)=2=\type(R)-1$.
Hence $R$ is almost Gorenstein. 

%Then we can take $(f_1, f_2)$ with $f_1\in R_1$ and $f_2\in R_2$ as 
%a minimal reduction of $\m_R$. Since $h^0(K_C - D)=0, h^0(K_C)=2$ and $h^0(K_C+D)=3$, 
%$K_R$ is generated by $2$ elements of degree $0$ and $1$ element of degree $1$ and 
%$\type(R)=3$.  Also, $h^0(nD) = 2n-1$ for $n\ge 1$, 
%$P(K_R/wR,t)=(1+t)/(1-t)$. Hence $e_0(K_R/wR)=2=\type(R)-1$ and $R$ is almost Gorenstein. 

\item 
Assume $\deg D = 3$. Then  $h^0(D)=2$. 
\begin{enumerate}
\item If $\cO_C(D)$ is not generated and if  $P$ is a base point of $\cO_C(D)$, 
then $H^0(D)= H^0(D-P)$. Since $h^0(D-P)=2=\deg(D-P)$, we may assume that $D = K_C+P$. 
Then $R$ is almost Gorenstein by Theorem \ref{p:K+p}.     

\item If  $\cO_C(D)$ is generated, then 
$H^0(\cO_C(D -K_C))=0$, 
since otherwise 
$D = K_C +P$ for some point $P\in C$ and then $\cO(D)$ is not generated. 
By Riemann-Roch Theorem, we also have $H^1(\cO_C(D -K_C))=0$.
Hence $K_R$ is generated by degree $0$ elements by 
\lemref{l:generation} (1).
Hence $R$ is level and not almost Gorenstein by \exref{level=nAG}.    
\end{enumerate}

{
\item If $\deg D \ge 4$, then $\cO_C(D)$ is generated and if $D\not\sim 2K_C$, 
then since  $H^1(\cO_C(D -K_C))=0$, $R$ is level by Lemma \ref{l:generation}
 and  not almost Gorenstein by  \exref{level=nAG}. If $D\sim 2K_C$,  
$R\cong (R')^{(2)}$, where $R' = k[X,Y,Z]/(X^2 - g_6(Y,Z))$, where $g_6(Y,Z)$ is 
a homogeneous polynomial of degree $6$ with no multiple roots. 
Then we see that $\type(R) =3$
 (see the proof of \lemref{l:generation} (1))
 and since $e_0(U) = 2$, $R$ is almost Gorenstein.}     
\qedhere
\end{enumerate}
\end{proof}

\section{Examples of determinantal singularities}\label{s:det}

In this section, we provide some examples under common settings as follows.
Let $S=k[s,u,v,w]$ be the polynomial ring and let $I_2(M) \subset S$ denote the ideal generated by $2\times 2$ minors of a $2\times 3$ matrix $M$ with entries in $S$. 
Note that we think $S$ a graded ring, but the degrees of the variables are not necessarily $1$.
Let $R=S/I_2(M)$. 
In the following, every ring $R$ we treat is a two-dimensional normal graded ring; in each case, the normality of $R$ will follows from that $R$ has only isolated singularity
 by Serre's criterion. 
Therefore, $R$ can be expressed as $R(C,D)=\bigoplus_{n\ge 0}H^0(C, \cO_C(D^{(n)}))$ of \proref{Tomari}. 
Let $\phi\: Y \to \spec R$ be a resolution of singularities. 
From Proposition 2.2 and Theorem 3.3 of \cite{Wt}, we have the isomorphisms 
\[
H_{\m_R}^2(R)\cong \bigoplus_{n\in \Z}H^1(C, \cO_C(D^{(n)})),
\quad 
R^1\phi_*\cO_Y\cong  \bigoplus_{n\ge 0} H_{\m_R}^2(R)_n
\]
of graded $R$-modules.
%{\color{red} 
As noted in Notation-Proposition \ref{Not6.1},
%}
% By \cite[Remark 2.10]{Wt}, the $a$-invariant (\cite{GW}) of $R$ is expressed as 
\[
a(R)=\max \defset{n \in Z}{H_{\m_R}^2(R)_n \ne 0}.
\]
We define  Poincar\'e series $P^0(R,t)$, $P^1(R,t)$ of $R$ by 
\[
P^i(R,t)=\sum_{n\ge 0}h^i(C, \cO_C(D^{(n)}))t^n \ \text{for $i=0,1$}.
\]
Then we have the following:
\[
a(R) = \deg P^1(R,t), \ g=h^1(\cO_C) = P^1(R,0), \ p_g(A) = \sum_{n=0}^{a(R)} \ell_R(H_{\m_R}^2(R)_n)= P^1(R,1).
\]
Let $P^{\chi}(R,t)=P^0(R,t)-P^1(R,t)$ and  $d_n=\deg D^{(n)}$.
By the Riemann-Roch theorem, we have 
\[
P^{\chi}(R,t)=\sum_{n\ge 0}\chi(\cO_C(D^{(n)}))t^n
=\sum_{n\ge 0}(1-g+d_n)t^n.
\]
Our graded ring $R$ also satisfies 
the condition  that $R/sR$ is a semigroup ring. Hence, it is easy to compute $P^0(R,t)$ explicitly.

For a rational function $f(t)=p(t)+r(t) /q(t)$, where $p, q, r$ are polynomials with $\deg r<\deg q$, we write $\pol(f)=p(t)$.
Note that for polynomials $p,q,r$ with
 $\deg r<\deg q$, 
we have $\pol( (p+r)/q ) = \pol(p/q)$.

\begin{prop}[cf. {\cite[4.8]{o.pg-splice}}]
 $\pol(P^0(R,t)) =  P^1(R,t)$.
\end{prop}
\begin{proof}
Since $1-g+d_n$ is a quasi-polynomial
function of $n$,  
there exist polynomials $r(t)$ and $q(t)$ such that $P^{\chi}(R,t)=r(t)/q(t)$ and $\deg r < \deg q$ by \cite[\S 4.4]{stan.enum}.
Hence 
the assertion follows from the equality $P^0(R,t)=P^1(R,t)+P^{\chi}(R,t)$. 
\end{proof}

Let $A$ be the localization of $R$ with respect to the maximal ideal $\m_R$.

%{\color{red}
\begin{ex}\label{ex:nrmpg}
We show an example that $A$ is non-rational and $\m$ is $p_g$-ideal. 
Let $M=\begin{pmatrix} u & v & w + s^{5} \\ v & w & u^2
\end{pmatrix}$.
 Then  $R=S/I_2(M)$ is a graded ring with $s\in R_1$, $u\in R_{3}$, $v\in R_{4}$, $w\in R_{5}$.
Since $R/sR = k[\gi{3,4,5}]$, $R$ is not Gorenstein and $P^0(R,t)= P^0(R/sR,t)/(1-t)=\left(1+t^{3}/(1-t)\right)/(1-t)$. 
Hence 
\[
P^1(R,t) = \pol\left(\frac{t^{3}-1}{(1-t)^2}\right)
 = \pol\left(\frac{t^{2}+t+1}{t-1}\right)
=2+t.
\]
Therefore, $g=g(C)=2$, 
$a(R) = 1$, 
$p_g(R)=P^1(R,1)=2+1=3$; in particular, $A$ is neither rational nor elliptic.
We also have $P^{\chi}(R,t) = P^0(R,t)-P^1(R,t) = (2 t-1)/(1-t)^2=\sum_{n\ge 0}(n-1)t^n$, and hence $d_n=n$.
This implies that  $D$ is a one-point  divisor.
Hence $\m$ is a $p_g$-ideal by \proref{p:mpg}.
\end{ex}
%}

\begin{ex}
We construct a 
non-elliptic  
non-Gorenstein singularity $A$ which is almost Gorenstein and $\m$ is not $p_g$-ideal. 
Let $M=\begin{pmatrix} u^2 & v^2+s^{10} & w \\ v & w & u
\end{pmatrix}$.
 Then  $R$ is a graded ring with $s\in R_1$, $u\in R_{4}$, $v\in R_{5}$, $w\in R_{7}$, and $R/sR = k[\gi{4,5,7}]$.
We see that  $R$ is not Gorenstein, and $R/sR$ is almost Gorenstein since the semigroup $\gi{4,5,7}$ is almost symmetric (see \cite{BF}).
Therefore, $A$ is also almost Gorenstein by \thmref{t:Hpg}.
We have $P^0(R,t)= \left(1+t^4+t^5+t^{7}/(1-t)\right)/(1-t)$, $P^1(R,t)=4 + 3 t + 2 t^2 + t^3 + t^4 + t^5$ and $P^{\chi}(R,t)= (4t-3)/(1-t)^2=\sum_{n\ge 0}(n-3)t^n$.
Hence we obtain that $g=4$, $a(R) = 5$, $p_g(R)=12$, $d_n=n$, 
$D$ is linearly equivalent to a 
one-point  
divisor, and hence $\m$ is not a $p_g$-ideal by  \proref{p:mpg}.
\end{ex}

\begin{ex}
We construct a singularity $A$ which is not almost Gorenstein. 
Let  $M=\begin{pmatrix} u^2 & v & w -s^8 \\ v-s^7 & w & u^3
\end{pmatrix}$.
 Then  $R$ is a graded ring with $s\in R_1$, $u\in R_{3}$, $v\in R_{7}$, $w\in R_{8}$, and $R/sR = k[\gi{3,7,8}]$.
%{\color{red} 
We have $P^0(R,t)= \left(1+t^3+t^{6}/(1-t)\right)/(1-t)
=(1+t^7+t^8) /(1-t)(1-t^3)$,  
$P^1(R,t)=4 + 3 t + 2 t^2 + 2 t^3 + t^4$ and $P^{\chi}(R,t)= (4t-3)/(1-t)^2=\sum_{n\ge 0}(n-3)t^n$.
Hence we obtain that $g=4$, $a(R) = 4$, $p_g(R)=12$,
$D$ is linearly equivalent to a one-point divisor.
If we put $H = \gi{3,7,8}$, then 
we see that $F(H)=5, g = 4$ and $\type(H) = 2$. Thus $H$ is not almost symmetric since 
$2g = 8 > F(H) + \type(H)=7$.
Hence $A$ is not almost Gorenstein by \thmref{t:Hpg}.
%}
\end{ex}

%{\color{red} 
\begin{ex}\label{ex:d=1}
We construct an elliptic singularity $A$ which is not Gorenstein.
Let 
$M=\begin{pmatrix} u & v & w  \\ v & w & u^2-s^3
\end{pmatrix}$.
 Then  $R$ is a graded ring with $s\in R_2$, $u\in R_{3}$, $v\in R_{4}$, $w\in R_{5}$, and $R/sR = k[\gi{3,4,5}]$.
Hence $R$ is not Gorenstein.
Let us compute $P^0(R,t)$ using a graded free resolution.
By the definition of $R$, we have the resolution
\[
0 \to S(-13)\oplus S(-14) \to S(-10)\oplus S(-9) \oplus S(-8) 
\to S \to R \to 0.
\]
Hence we have 
\[
P^0(R,t)= \frac{1-(t^8+t^9+t^{10})+(t^{13}+t^{14})}{(1-t^2)(1-t^3)(1-t^4)(1-t^5)}
=\frac{1-t+t^3}{(1-t)^2 (1+t)}, \ P^1(R,t)=1.
\]
Also, we have $P^{\chi}(R,t)=(t^2+t^3)/(1-t^2)^2$.
Therefore,  $g=p_g(R)=1$, 
$a(R) = 0$, and $d_{2n}=n$ for $n\ge 0$.
Hence $A$ is an elliptic singularity.
We will identify the $\Q$-divisor $D$.
Let $D=Q-\sum_{i=1}^d q_iP_i$, where $Q$ is a divisor, $P_i\in C$ are distinct points, and $q_i\in \Q$ satisfies $0<q_1\le \cdots \le q_d<1$.
Since $d_1=\chi(\cO_C(D^{(1)})) = 0$, we have $\deg Q=d$. 
Hence $\sum_{i=1}^d(2 - \ce{ 2q_i})=d_2=1$.
Therefore, $q_1\le 1/2 <q_2$.
For any $n\in \Z_{\ge 1}$ such that $nq_i\in \Z$ for $i=1, \dots, d$, we have $n=d_{2n}=\sum_{i=1}^d(2n-2nq_i)$. Since $q_1\le 1/2$, this implies that $d=1$ and $q_1=1/2$.
Note that $Q\not\sim P_1$ because $H^0(C,\cO_C(D^{(1)}))=0$, and the resolution graph of $A$ is the same as that of $k[[x,y,z]]/(x^2+y^3+z^{12})$; thus we obtain the case $d=1$ of \exref{ex:ellL=2}.
\end{ex}
%}

\providecommand{\bysame}{\leavevmode\hbox to3em{\hrulefill}\thinspace}
\providecommand{\MR}{\relax\ifhmode\unskip\space\fi MR }
\providecommand{\MRhref}[2]{%
  \href{http://www.ams.org/mathscinet-getitem?mr=#1}{#2}
}
\providecommand{\href}[2]{#2}

\end{document}